\documentclass[12pt]{amsart}

\usepackage{amsmath,amssymb,amsfonts,amsxtra,hyperref,fullpage,xypic,centernot,setspace}
\usepackage{soul}
\setuldepth{zarikian}

\nocite{*}

\theoremstyle{plain}
\newtheorem*{theorem*}{Theorem}
\newtheorem{theorem}{Theorem}[section]
\newtheorem{lemma}[theorem]{Lemma}

\newtheorem{corollary}[theorem]{Corollary}
\newtheorem{example}[theorem]{Example}
\theoremstyle{definition}
\newtheorem{definition}[theorem]{Definition}
\newtheorem{remark}[theorem]{Remark}
\newtheorem{question}[theorem]{Question}

\DeclareMathOperator{\bbC}{\mathbb{C}}
\DeclareMathOperator{\bbE}{\mathbb{E}}
\DeclareMathOperator{\bbF}{\mathbb{F}}
\DeclareMathOperator{\bbH}{\mathbb{H}}
\DeclareMathOperator{\bbK}{\mathbb{K}}
\DeclareMathOperator{\bbM}{\mathbb{M}}

\DeclareMathOperator{\bbT}{\mathbb{T}}
\DeclareMathOperator{\bbZ}{\mathbb{Z}}
\DeclareMathOperator{\A}{\mathcal{A}}
\DeclareMathOperator{\Ad}{\operatorname{Ad}}
\DeclareMathOperator{\Aut}{\operatorname{Aut}}
\DeclareMathOperator{\B}{\mathcal{B}}
\DeclareMathOperator{\C}{\mathcal{C}}

\DeclareMathOperator{\D}{\mathcal{D}}
\DeclareMathOperator{\E}{\mathcal{E}}

\DeclareMathOperator{\fix}{\operatorname{fix}}
\renewcommand{\H}{\mathcal{H}}

\DeclareMathOperator{\id}{\operatorname{id}}
\DeclareMathOperator{\Ind}{\operatorname{Ind}}
\DeclareMathOperator{\J}{\mathcal{J}}

\DeclareMathOperator{\M}{\mathcal{M}}

\let\slasho=\o
\renewcommand{\o}{\overline}
\DeclareMathOperator{\one}{\mathbf{1}}
\DeclareMathOperator{\orb}{\operatorname{orb}}

\DeclareMathOperator{\Proj}{\operatorname{Proj}}

\DeclareMathOperator{\ran}{\operatorname{ran}}
\renewcommand{\span}{\operatorname{span}}
\DeclareMathOperator{\U}{\mathcal{U}}
\DeclareMathOperator{\W}{\mathcal{W}}

\begin{document}

\nocite{*}

\title{Unique Pseudo-Expectations for Dynamical $C^*$-Inclusions}
\author{Vrej Zarikian}
\address{U. S. Naval Academy, Annapolis, MD 21402}
\email{zarikian@usna.edu}
\date{\today}
\begin{abstract}
Let $(\A,G,\alpha)$ be a $C^*$-dynamical system. Unique extension properties for the $C^*$-inclusion $\A \subseteq \A \rtimes_{\alpha,r} G$ have been studied extensively. In this paper, we investigate unique extension properties for some other natural ``dynamical'' $C^*$-inclusions, namely:
\begin{itemize}
\item $C_r^*(G) \subseteq \A \rtimes_{\alpha,r} G$;
\item $\A^G \subseteq \A$, where $\A^G$ is the fixed-point subalgebra;
\item $Z(\A)^c \subseteq \A \rtimes_{\alpha,r} G$, where $Z(\A)^c = Z(\A)' \cap (\A \rtimes_{\alpha,r} G)$ is the relative commutant of the center.
\end{itemize}
Our hope is that a larger selection of examples will enable progress on some lingering open problems, in particular (1) the relationship between aperiodicity and the unique pseudo-expectation property and (2) the relationship between the faithful unique pseudo-expectation property and norming. 
\end{abstract}
\maketitle

\tableofcontents

\section{Introduction} \label{intro}

Let $(\A,G,\alpha)$ be a $C^*$-dynamical system. Unique extension properties for the $C^*$-inclusion $\A \subseteq \A \rtimes_{\alpha,r} G$ have been studied extensively, and are now completely understood. Indeed, we have the following summary result (see Section \ref{notation} below for any unfamiliar notation or terminology):

\begin{theorem}
Let $(\A,G,\alpha)$ be a $C^*$-dynamical system. Then the following statements hold:
\begin{enumerate}
\item[i.] $\A \subseteq \A \rtimes_{\alpha,r} G$ has the \textbf{pure extension property} (PEP) $\iff$ $\hat{\alpha}:G \curvearrowright \mathcal{\widehat{A}}$ is pointwise free. \cite[Corollary 2.5]{Zarikian2019b}
\item[ii.] $\A \subseteq \A \rtimes_{\alpha,r} G$ has the \textbf{almost extension property} (AEP) $\iff$ $\hat{\alpha}:G \curvearrowright \mathcal{\widehat{A}}$ is essentially free (i.e., pointwise free on a dense set). \cite[Corollary 2.6]{Zarikian2019b}
\item[iii.] $\A \subseteq \A \rtimes_{\alpha,r} G$ has a \textbf{unique pseudo-expectation} $\iff$ $\alpha:G \curvearrowright \A$ is pointwise properly outer $\iff$ $\A \subseteq \A \rtimes_{\alpha,r} G$ is aperiodic. \cite[Theorem 3.5]{Zarikian2019a} and \cite[Proposition 2.4]{KwasniewskiMeyer2022}
\item[iv.] $\A \subseteq \A \rtimes_{\alpha,r} G$ has a \textbf{unique conditional expectation} $\iff$ $\alpha:G \curvearrowright \A$ is pointwise freely acting $\iff$ $\A' \cap (\A \rtimes_{\alpha,r} G)=Z(\A)$. \cite[Theorem 3.2]{Zarikian2019a}
\end{enumerate}
\end{theorem}

Of course the same $C^*$-dynamical system $(\A,G,\alpha)$ gives rise to other important $C^*$-inclusions, for example $C_r^*(G) \subseteq \A \rtimes_{\alpha,r} G$ and $\A^G \subseteq \A$, which can exhibit very different behavior than $\A \subseteq \A \rtimes_{\alpha,r} G$. These less-studied dynamical $C^*$-inclusions have received more attention lately \cite{AmrutamKalantar2020,EchterhoffRordam2024,Rordam2023,Zarikian2025}, and the purpose of this paper is to investigate their unique extension properties, in particular to determine when they have a unique pseudo-expectation. Intrinsic interest notwithstanding, our main motivation is to increase the available supply of $C^*$-inclusions with the unique pseudo-expectation property, with an eye toward making progress on some fundamental open problems, such as the relationship between the unique pseudo-expectation property and aperiodicity (see Section \ref{hierarchy} below), as well as the connection between the faithful unique pseudo-expectation property and norming \cite{PittsSmithZarikian2024}.\\

In the remainder of Section \ref{intro}, we establish notation and discuss some preliminaries. In Section \ref{general_results}, we prove some unique pseudo-expectation results for $C^*$-inclusions which are not necessarily dynamical in nature. In particular we consider finite-dimensional inclusions, finite-index inclusions (in the sense of Watatani), and corner inclusions. These more general results facilitate our analysis of dynamical $C^*$-inclusions in Section \ref{dynamical_results}, where the reader will find our main results. There we consider
\begin{itemize}
\item the inclusion $C_r^*(G) \subseteq \A \rtimes_{\alpha,r} G$, for which uniqueness of pseudo-expectations is related to unique ergodicity of $\alpha:G \curvearrowright \A$;
\item the inclusion $\A^G \subseteq \A$ for $|G|<\infty$, which is a corner of the inclusion $\A \rtimes_{\alpha,r} G \subseteq \A \otimes B(\ell^2(G))$, for which Watatani's theory of finite-index inclusions is an incisive tool; 
\item and the inclusion $Z(\A)^c \subseteq \A \rtimes_{\alpha,r} G$, where $Z(\A)^c$ (the relative commutant of the center) serves as a technically less onerous proxy for $C^*(\A \cup \A^c)$ (the $C^*$-algebra generated by $\A$ and its relative commutant $\A^c$).
\end{itemize}
Although we make good progress in each of the aforementioned directions, there is still substantial room for our results to be strengthened and/or generalized, and the open problems mentioned at the end of the previous paragraph remain unsolved.

\subsection{Notation and Terminology} \label{notation}

Most ambient $C^*$-algebras in this paper will be unital. If $\A$ is a $C^*$-algebra, then $Z(\A) = \A' \cap \A$ will denote the \emph{center}, $U(\A)$ the \emph{unitary group}, and $\Proj(\A)$ the \emph{orthogonal projections}. The \emph{spectrum} of $\A$ will be written $\widehat{\mathcal{A}}$. It consists of the non-zero irreducible representations of $\A$ modulo unitary equivalence, and can be equipped with a locally compact (but not necessarily Hausdorff) topology. Every $C^*$-algebra $\A$ is contained in a unique minimal injective operator system $I(\A)$, called the \emph{injective envelope}, which is actually a $C^*$-algebra containing $\A$ as a $C^*$-subalgebra \cite[Theorem 4.1]{Hamana1979}. The \emph{automorphism group} of a $C^*$-algebra $\A$ will be denoted $\Aut(\A)$. We say that $\alpha \in \Aut(\A)$ is \emph{inner} if there exists $u \in U(\A)$ such that $\alpha=\Ad(u)$, and \emph{outer} otherwise. Following \cite{ChodaKasaharaNakamoto1972}, we say that $d \in \A$ is a \emph{dependent element} for  $\alpha \in \Aut(\A)$ if $\alpha(a)d=da$ for all $a \in \A$. If $\alpha \in \Aut(\A)$ has no non-zero dependent elements, then we say that it is \emph{freely acting}. If $\alpha \in \Aut(\A)$, then the formula $\hat{\alpha}([\pi]) = [\pi \circ \alpha^{-1}]$ defines a homeomorphism of $\widehat{\mathcal{A}}$. We say that $\hat{\alpha}$ is \emph{free} if $\fix(\hat{\alpha})=\emptyset$ and \emph{topologically free} if $\fix(\hat{\alpha})^\circ=\emptyset$. Every $\alpha \in \Aut(\A)$ extends uniquely to $I(\alpha) \in \Aut(I(\A))$. We say that $\alpha$ is \emph{properly outer} if there does not exist $0 \neq p \in \Proj(Z(I(\A)))$ such that $I(\alpha)(p)=p$ and $I(\alpha)|_{I(\A)p}$ is inner.\\

For $G$ a discrete group, $\lambda:G \to U(B(\ell^2(G))$ will denote the \emph{left regular representation},
\[
    \lambda_g\delta_h = \delta_{gh}, ~ h \in G,
\]
and $\rho:G \to U(B(\ell^2(G)))$ will denote the \emph{right regular representation},
\[
    \rho_g\delta_h = \delta_{hg^{-1}}, ~ h \in G.
\]
The \emph{reduced group $C^*$-algebra} of $G$ is
\[
    C_r^*(G) = \overline{\span}\{\lambda_g: g \in G\} \subseteq B(\ell^2(G)),
\]
and the \emph{right and left group von Neumann algebras} of $G$ are
\[
    L(G) = \{\lambda_g: g \in G\}'' \subseteq B(\ell^2(G)) \text{ and } R(G) = \{\rho_g: g \in G\}'' \subseteq B(\ell^2(G)),
\]
respectively.\\

A \emph{$C^*$-dynamical system} $(\A,G,\alpha)$ will consist of a $C^*$-algebra $\A$, a discrete group $G$, and a homomorphism $\alpha:G \to \Aut(\A)$. As is customary, we will write $\alpha:G \curvearrowright \A$ instead of $\alpha:G \to \Aut(\A)$ and $\alpha_g$ instead of $\alpha(g)$. We denote the \emph{reduced crossed product} by $\A \rtimes_{\alpha,r} G$. It may identified with the $C^*$-subalgebra of $\A^{**} \overline{\otimes} B(\ell^2(G))$ (or just $\A \otimes B(\ell^2(G))$ if $|G|<\infty$) generated by the operators
\[
    \sum_g \left(\alpha_g^{-1}(a) \otimes E_{g,g}\right), ~ a \in \A,
\]
and
\[
    1 \otimes \lambda_g = 1 \otimes \sum_h E_{gh,h}, ~ g \in G,
\]
where $E_{g,h} \in B(\ell^2(G))$ is the rank-one operator such that $E_{g,h}\delta_h=\delta_g$. For most purposes it is safe to think of $\A \rtimes_{\alpha,r} G$ as a certain $C^*$-completion of $C_c(G,\A) := \span\{a_g\lambda_g: a_g \in \A, g \in G\}$, subject to the relations
\[
    (a_g\lambda_g)(a_h\lambda_h) = a_g\alpha_g(a_h)\lambda_{gh} \text{ and } (a_g\lambda_g)^* = \alpha_{g^{-1}}(a_g^*)\lambda_{g^{-1}}.
\]
We identify $\A$ with $\{a_e\lambda_e: a_e \in \A\} \subseteq \A \rtimes_{\alpha,r} G$ and note that there exists a faithful conditional expectation $\bbE_{\A}:\A \rtimes_{\alpha,r} G \to \A$ such that $\bbE_{\A}\left(\sum_g a_g\lambda_g\right)=a_e$ for all $\sum_g a_g\lambda_g \in C_c(G,\A)$. Lastly,
\[
    \A^G = \{a \in \A: (\forall g \in G) ~ \alpha_g(a)=a\}
\]
will denote the \emph{fixed-point algebra}.\\

A \emph{$C^*$-inclusion} is an inclusion $\A \subseteq \B$ of $C^*$-algebras such that $1_{\B}=1_{\A}$. If $\A \subseteq \B$ is a $C^*$-inclusion, then we write $\A^c = \A' \cap \B$ for the \emph{relative commutant}. The $C^*$-inclusion $\A \subseteq \B$ has
\begin{itemize}
\item the \emph{pure extension property} (PEP) if every pure state $\phi \in PS(\A)$ extends uniquely to a pure state $\tilde{\phi} \in PS(\B)$;
\item the \emph{almost extension property} (AEP) if the set $PS(\A \uparrow \B)$ of pure states on $\A$ which extend uniquely to pure states on $\B$ is weak* dense in $PS(\A)$ \cite[p. 265]{NagyReznikoff2014};
\item a \emph{unique pseudo-expectation} if there exists a unique ucp (unital completely positive) map $\Psi:\B \to I(\A)$ such that $\Psi(a)=a$ for all $a \in \A$ \cite[Definition 1.3]{Pitts2017};
\item a \emph{unique conditional expectation} if there exists a unique ucp map $E:\B \to \A$ such that $E(a)=a$ for all $a \in \A$.
\end{itemize}
Following \cite[Definition 2.3]{KwasniewskiMeyer2022}, we say that the $C^*$-inclusion $\A \subseteq \B$ is \emph{aperiodic} if
\[
    (\forall b \in \B)(\forall \D \in \bbH(\A))(\forall \varepsilon > 0)(\exists d \in \D_+^1)(\exists a \in \A)
    ~ \|dbd-a\| < \varepsilon.
\]
Here $\bbH(\A)$ denotes the non-zero hereditary subalgebras of $\A$ and $\D_+^1$ denotes the norm-one positive elements of $\D$.

\subsection{The Hierarchy of Unique Extension Properties} \label{hierarchy}

The following theorem summarizes the hierarchy of unique extension properties for general $C^*$-inclusions.

\begin{theorem}[\cite{KwasniewskiMeyer2022,PittsSmithZarikian2024}] \label{hierarchy}
For a $C^*$-inclusion $\A \subseteq \B$, consider the statements:
\begin{enumerate}
\item[i.] $\A \subseteq \B$ has the pure extension property (PEP);
\item[ii.] $\A \subseteq \B$ has the almost extension property (AEP);
\item[iii.] $\A \subseteq \B$ is aperiodic;
\item[iv.] $\A \subseteq \B$ has a unique pseudo-expectation;
\item[v.] $\A \subseteq \B$ has at most one conditional expectation.
\end{enumerate}
Then
\[
    (i) \implies (ii) \implies (iii) \implies (iv) \implies (v).
\]
And if $\B$ is separable, then
\[
    (ii) \iff (iii).
\]
\end{theorem}

\begin{proof}
The implications $(i) \implies (ii)$ and $(iv) \implies (v)$ are trivial. The implication $(ii) \implies (iv)$, which pre-dated the notion of aperiodicity, is \cite[Theorem A.1.2]{PittsSmithZarikian2024}. The implication $(ii) \implies (iii)$ (and the equivalence $(ii) \iff (iii)$ when $\B$ is separable) is \cite[Theorem 5.5]{KwasniewskiMeyer2022}. Finally, the implication $(iii) \implies (iv)$ is \cite[Theorem 3.6]{KwasniewskiMeyer2022}.
\end{proof}

As mentioned in the introduction, this paper was partially motivated by the following open problem.

\begin{question}[{\cite[Remark 7.5]{KwasniewskiMeyer2022}}] \label{open}
Is aperiodicity equivalent to the unique pseudo-expectation property? 
\end{question}

\begin{remark} \label{comments}
\phantom{}
\begin{itemize}
\item If Question \ref{open} had an affirmative answer, then one would have a characterization of the unique pseudo-expectation property in terms of the structure of the inclusion, without ever having to mention the injective envelope. 
\item Given that aperiodicity is equivalent to the AEP when the ambient $C^*$-algebra is separable, one might think that the correct question is whether or not aperiodicity is equivalent to the AEP. But there are aperiodic abelian inclusions which fail the AEP \cite[Theorem 6.2 and Example 5.1]{Zarikian2022}.
\end{itemize}
\end{remark}

\subsection{Injective Envelopes of Crossed Products} \label{inj}

In order to analyze pseudo-expectations for inclusions of the form $C_r^*(G) \subseteq \A \rtimes_{\alpha,r} G$, we need some insight into $I(C_r^*(G))$, the injective envelope of $C_r^*(G)$. Since $C_r^*(G) = \bbC \rtimes_r G$, it suffices to understand $I(\A \rtimes_{\alpha,r} G)$, the injective envelope of $\A \rtimes_{\alpha,r} G$. Although an exact description is not available, the following bounds often carry the day:
\[
    I_G(\A) \rtimes_{I_G(\alpha),r} G \subseteq I(\A \rtimes_{\alpha,r} G) \subseteq I_G(\A) \overline{\rtimes}_{I_G(\alpha)} G.
\]
Here $I_G(\A)$ is the \emph{$G$-equivariant injective envelope} of $\A$ equipped with its canonical action $I_G(\alpha):G \curvearrowright I_G(\A)$ and $\overline{\rtimes}$ denotes the \emph{monotone complete crossed product}. Specializing to the case $\A=\bbC$, we get
\[
    C(\partial_FG) \rtimes_{\sigma,r} G \subseteq I(C_r^*(G)) \subseteq C(\partial_FG) \overline{\rtimes}_\sigma G,
\]
where $\partial_FG$ is the \emph{Furstenberg boundary} of $G$ and $\sigma:G \curvearrowright C(\partial_FG)$ is the canonical action. In the remainder of this section we provide enough background for the reader to make sense of this paragraph.\\

For any $C^*$-dynamical system $(\A,G,\alpha)$, Hamana proves the existence and uniqueness of the $G$-injective envelope $I_G(\A)$ \cite[Theorem 2.5]{Hamana1985}. This means that $I_G(\A)$ is a minimal injective containing $\A$ in the category whose objects are operator systems equipped with a $G$-action and whose morphisms are $G$-equivariant ucp maps. We denote the $G$-action on $I_G(\A)$ by $I_G(\alpha):G \curvearrowright I_G(\A)$ and note that it extends $\alpha:G \curvearrowright \A$. A priori $I_G(\A)$ is just an operator system, but it becomes a $C^*$-algebra with respect to the Choi-Effros product. Since $I_G(\A)$ is injective in the traditional (non-equivariant) sense, there are operator system inclusions $\A \subseteq I(\A) \subseteq I_G(\A)$ \cite[Remark 2.3]{Hamana1985}. With some care, one can arrange that these inclusions are actually $G$-equivariant $*$-homomorphisms, where $I(\A)$ is equipped with the unique $G$-action $I(\alpha):G \curvearrowright I(\A)$ extending $\alpha:G \curvearrowright \A$ \cite[Section 3]{Hamana1985}. Somewhat surprisingly, $\bbC$ is not always $G$-injective, so that $I_G(\bbC) \neq \bbC$ in general. In fact, $I_G(\bbC)=\bbC$ if and only if $G$ is amenable \cite[Remark 3.8]{Hamana1985}. Instead, $I_G(\bbC)=C(\partial_FG)$, where $\partial_FG$ is the Furstenberg boundary of $G$ \cite[Theorem 3.11]{KalantarKennedy2017}. Since $C(\partial_FG)$ is injective, $\partial_FG$ is an extremally disconnected compact Hausdorff space. If $\sigma:G \curvearrowright C(\partial_FG)$ denotes the canonical action, then there is a corresponding continuous action $\hat{\sigma}:G \curvearrowright \partial_FG$, which is always minimal \cite[Proposition 3.4]{KalantarKennedy2017}.\\

In \cite[Section 3]{Hamana1982II}, Hamana introduces the monotone complete crossed product $M(\A,G)$ corresponding to a $C^*$-dynamical system $(\A,G,\alpha)$, where $\A$ is a monotone complete $C^*$-algebra. Since injective $C^*$-algebras are monotonically complete, this construction applies to $(I(\A),G,I(\alpha))$ or $(I_G(\A),G,I_G(\alpha))$, for an arbitrary $C^*$-dynamical system $(\A,G,\alpha)$. Of course von Neumann algebras are monotonically complete as well. If $\A$ is a von Neumann algebra, then $M(\A,G) = \A \overline{\rtimes}_\alpha G$, the von Neumann-algebraic crossed product. Because of this, no ambiguity arises if we write $\A \overline{\rtimes}_\alpha G$ for $M(\A,G)$ in general, which we prefer for reasons of notational intuitiveness. As one would expect, $\A \overline{\rtimes}_\alpha G$ is a monotone complete $C^*$-algebra admitting a covariant representation $(\pi,u)$ of $(\A,G,\alpha)$ such that
\[
    \left\|\sum_g \pi(a_g)u_g\right\|_{\A \overline{\rtimes}_\alpha G}=\left\|\sum_g a_g\lambda_g\right\|_{\A \rtimes_{\alpha,r} G}
\]
for all $a \in C_c(G,\A)$. In other words, we are free to imagine $\A \rtimes_{\alpha,r} G \subseteq \A \overline{\rtimes}_\alpha G$. The canonical conditional expectation $\bbE_{\A}:\A \rtimes_{\alpha,r} G \to \A$ extends to a faithful conditional expectation $\overline{\bbE}_{\A}:\A \overline{\rtimes}_{\alpha} G \to \A$, which is $G$-equivariant in the sense that $\overline{\bbE}_{\A}(\lambda_gx\lambda_g^*)=\alpha_g(\overline{\bbE}_{\A}(x))$ for all $x \in \A \overline{\rtimes}_\alpha G$ and $g \in G$. Every $x \in \A \overline{\rtimes}_\alpha G$ has a formal ``Fourier series'' representation $x \sim \sum_g x_g\lambda_g$, where $x_g := \overline{\bbE}_{\A}(x\lambda_g^*)$ for all $g \in G$. As usual, $x=0$ if and only if $x_g=0$ for all $g \in G$, although one cannot appeal to norm density of $C_c(G,\A)$ for the proof. Formulas such as $ax \sim \sum_g ax_g\lambda_g$, $a \in \A$, and $x\lambda_h \sim \sum_g x_g\lambda_{gh}$, $h \in G$, are easily verified by comparing ``Fourier coefficients''. Finally, $\A \overline{\rtimes}_\alpha G$ is injective if and only if $\A$ is $G$-injective \cite[Lemma 3.2]{Hamana1985}. Thus $I_G(\A) \overline{\rtimes}_{I_G(\alpha)} G$ is an injective $C^*$-algebra containing $I_G(\A) \rtimes_{I_G(\alpha),r} G$ and therefore $\A \rtimes_{\alpha,r} G$. It follows that $I(\A \rtimes_{\alpha,r} G) \subseteq I_G(\A) \overline{\rtimes}_{I_G(\alpha)} G$ as \ul{operator systems}, although the $C^*$-product on $I(\A \rtimes_{\alpha,r} G)$ is not necessarily the restriction of the product on $I_G(\A) \overline{\rtimes}_{I_G(\A)} G$. On the other hand, Hamana proves that $I_G(\A) \rtimes_{I_G(\alpha),r} G \subseteq I(\A \rtimes_{\alpha,r} G)$ and that the restrictions of the products on $I(\A \rtimes_{\alpha,r} G)$ and $I_G(\A) \overline{\rtimes}_{I_G(\alpha)} G$ to $I_G(\A) \rtimes_{I_G(\alpha),r} G$ do coincide \cite[Theorem 3.4]{Hamana1985}.

\subsection{Decompositions of Equivariant Completely Bounded Maps}

As noted in the previous section, even though our interests lie in the non-equivariant setting, equivariant considerations show up naturally. In this section we prove equivariant versions of some classical decomposition results for completely bounded maps (Lemmas \ref{equiv_cb_decomp} and \ref{equiv_cp_decomp} below). These results are not surprising and probably known to experts in the field, but we could not find the exact statements we needed in the literature.

\begin{lemma}[{\cite[Remark 2.3]{Hamana1985}}] \label{canon_embed}
Let $(\A,G,\alpha)$ be a $C^*$-dynamical system. Then $\iota_{\A}:\A \to \ell^\infty(G,\A)$ defined by
\[
    \iota_{\A}(a) = (\alpha_h^{-1}(a))_{h \in G}
\]
is a $G$-equivariant unital $*$-monomorphism, where $\ell^\infty(G,\A)$ is equipped with the translation action $\tau:G \curvearrowright \ell^\infty(G,\A)$ defined by
\[
    \tau_g((a_h)_{h \in G}) = (a_{g^{-1}h})_{h \in G}.
\]
\end{lemma}

\begin{lemma}
Let $\A$ be $G$-injective. Then $M_n(\A)$ is $G$-injective.
\end{lemma}

\begin{proof}
By \cite[Remark 2.3]{Hamana1985}, $\A$ is $G$-injective if and only if $\A$ is injective and there exists a $G$-equivariant ucp map $P:\ell^\infty(G,\A) \to \A$ such that $P \circ \iota_{\A}=\id_{\A}$. But then $M_n(\A)$ is injective and $P^{(n)}:M_n(\ell^\infty(G,\A)) \to M_n(\A)$ is a $G$-equivariant ucp map such that $P^{(n)} \circ \iota_{\A}^{(n)} \to \id_{\A}^{(n)}$. Making the identification $M_n(\ell^\infty(G,\A))=\ell^\infty(G,M_n(\A))$ yields the result.
\end{proof}

\begin{lemma}[{equivariant version of \cite[Lemma 5.4.3]{EffrosRuanBook}}] \label{equiv_cb_decomp}
Let $(\A,G,\alpha)$, $(\B,G,\beta)$ be $C^*$-dynamical systems and $\phi:\A \to \B$ be a $G$-equivariant cb (completely bounded) map. If $\B$ is $G$-injective, then there exist $G$-equivariant cp (completely positive) maps $\phi_i:\A \to \B$, $1 \leq i \leq 4$, such that
\[
    \phi=(\phi_1-\phi_2)+i(\phi_3-\phi_4).
\]
\end{lemma}

\begin{proof}
We may assume that $\phi$ is completely contractive. Let
\[
    \mathcal{S}_{\A} 
    = \left\{\begin{bmatrix} \lambda 1_{\A} & a_1\\ a_2 & \mu 1_{\A} \end{bmatrix}: a_1, a_2 \in \A, ~ \lambda, \mu \in \bbC\right\} \subseteq M_2(\A)
\] 
be the Paulsen operator system and $\Phi:\mathcal{S}_{\A} \to M_2(\B)$ be the ucp map defined by
\[
    \Phi\left(\begin{bmatrix} \lambda1_{\A} & a_1\\ a_2^* & \mu1_{\A} \end{bmatrix}\right)
    = \begin{bmatrix} \lambda 1_{\B} & \phi(a_1)\\ \phi(a_2)^* & \mu 1_{\B} \end{bmatrix}
\]
\cite[Lemma 8.1]{PaulsenBook}. We note that $\Phi$ is $G$-equivariant, because
\begin{eqnarray*}
    \Phi\left(\alpha_g^{(2)}\left(\begin{bmatrix} \lambda1_{\A} & a_1\\ a_2^* & \mu1_{\A} \end{bmatrix}\right)\right)
    &=& \Phi\left(\begin{bmatrix} \alpha_g(\lambda1_{\A}) & \alpha_g(a_1)\\ \alpha_g(a_2^*) & \alpha_g(\mu1_{\A}) \end{bmatrix}\right)\\
    &=& \Phi\left(\begin{bmatrix} \lambda1_{\A} & \alpha_g(a_1)\\ \alpha_g(a_2)^* & \mu1_{\A} \end{bmatrix}\right)\\
    &=& \begin{bmatrix} \lambda1_{\B} & \phi(\alpha_g(a_1))\\ \phi(\alpha_g(a_2))^* & \mu1_{\B} \end{bmatrix}\\
    &=& \begin{bmatrix} \lambda1_{\B} & \beta_g(\phi(a_1))\\ \beta_g(\phi(a_2))^* & \mu1_{\B} \end{bmatrix}\\
    &=& \begin{bmatrix} \beta_g(\lambda1_{\B}) & \beta_g(\phi(a_1))\\ \beta_g(\phi(a_2)^*) & \beta_g(\mu1_{\B}) \end{bmatrix}\\
    &=& \beta_g^{(2)}\left(\begin{bmatrix} \lambda1_{\B} & \phi(a_1)\\ \phi(a_2)^* & \mu1_{\B} \end{bmatrix}\right)\\
    &=& \beta_g^{(2)}\left(\Phi\left(\begin{bmatrix} \lambda1_{\A} & a_1\\ a_2^* & \mu1_{\A} \end{bmatrix}\right)\right).
\end{eqnarray*}
Since $M_2(\B)$ is $G$-injective by the previous lemma, there exists a $G$-equivariant ucp map $\tilde{\Phi}:M_2(\A) \to M_2(\B)$ extending $\Phi$. Since $\tilde{\Phi}$ fixes $\bbC \oplus \bbC$, it is a ($\bbC \oplus \bbC$)-bimodule map \cite[Corollary 3.19]{PaulsenBook}. It follows that
\[
    \tilde{\Phi}\left(\begin{bmatrix} a_{11} & a_{12}\\ a_{21} & a_{22} \end{bmatrix}\right) = \begin{bmatrix} \psi_1(a_{11}) & \phi(a_{12})\\ \phi^*(a_{21}) & \psi_2(a_{22}) \end{bmatrix},
\]
where $\psi_1, \psi_2:\A \to \B$ are $G$-equivariant ucp maps. Now let $P:\A \to M_2(\A)$ be the $G$-equivariant cp map defined by
\[
    P(a) = \begin{bmatrix} a & a\\ a & a \end{bmatrix}. 
\]
Then
\[
    \phi = (\phi_1-\phi_2)+i(\phi_3-\phi_4),
\]
where
\[
    \phi_1 = \frac{1}{4}\begin{bmatrix} 1 & 1 \end{bmatrix}(\tilde{\Phi} \circ P)\begin{bmatrix} 1\\ 1 \end{bmatrix},
\]
\[
    \phi_2 = \frac{1}{4}\begin{bmatrix} 1 & -1 \end{bmatrix}(\tilde{\Phi} \circ P)\begin{bmatrix} 1\\ -1 \end{bmatrix},
\]
\[
    \phi_3 = \frac{1}{4}\begin{bmatrix} 1 & i \end{bmatrix}(\tilde{\Phi} \circ P)\begin{bmatrix} 1\\ -i \end{bmatrix},
\]
and
\[
    \phi_4 = \frac{1}{4}\begin{bmatrix} 1 & -i \end{bmatrix}(\tilde{\Phi} \circ P)\begin{bmatrix} 1\\ i \end{bmatrix}.
\]
\end{proof}

\begin{lemma}[{equivariant version of \cite[Lemma 5.1.6]{EffrosRuanBook}}] \label{equiv_cp_decomp}
Let $(\A,G,\alpha)$, $(\B,G,\beta)$ be $C^*$-dynamical systems and $\phi:\A \to \B$ be a $G$-equivariant cp map. If $\B$ is $G$-injective, then there exists a $G$-equivariant ucp map $\tilde{\phi}:\A \to \B$ such that $\phi(a)=\phi(1)^{1/2}\tilde{\phi}(a)\phi(1)^{1/2}$, $a \in \A$.
\end{lemma}

\begin{proof}
Suppose $\B \subseteq B(\H)$. By \cite[Lemma 5.1.6]{EffrosRuanBook}, there exists a ucp map $\psi:\A \to B(\H)$ such that
\[
    \phi(a)=\phi(1)^{1/2}\psi(a)\phi(1)^{1/2}, ~ a \in \A.
\]
Then $\psi^{(\infty)}:\ell^\infty(G,\A) \to \ell^\infty(G,B(\H))$ is a $G$-equivariant ucp map. Let $\iota_{\A}:\A \to \ell^\infty(G,\A)$ and $\iota_{\B}:\B \to \ell^\infty(G,B(\H))$ be the canonical $G$-equivariant unital $*$-monomorphisms (see Lemma \ref{canon_embed}). Since $\B$ is $G$-injective, there exists a $G$-equivariant ucp map $P:\ell^\infty(G,B(\H)) \to \B$ such that $P \circ \iota_{\B} = \id_{\B}$. Since $P|_{\iota_{\B}(\B)}$ is a $*$-homomorphism, Choi's Lemma implies
\[
    P(\iota_{\B}(b_1)x\iota_{\B}(b_2)) = b_1P(x)b_2, ~ x \in \ell^\infty(G,B(\H)), ~ b_1, b_2 \in \B.
\]
Now define a $G$-equivariant ucp map
\[
    \tilde{\phi} = P \circ \psi^{(\infty)} \circ \iota_{\A}:\A \to \B.
\]
Using the fact that
\[
    \beta_g(\phi(1)) = \phi(\alpha_g(1)) = \phi(1),
\]
we see that for all $a \in \A$,
\begin{eqnarray*}
    \phi(1)^{1/2}\tilde{\phi}(a)\phi(1)^{1/2}
    &=& \phi(1)^{1/2}P(\psi^{(\infty)}(\iota_{\A}(a)))\phi(1)^{1/2}\\
    &=& P(\iota_{\B}(\phi(1)^{1/2})\psi^{(\infty)}(\iota_{\A}(a))\iota_{\B}(\phi(1)^{1/2}))\\
    &=& P((\beta_g^{-1}(\phi(1)^{1/2}))_{g \in G}(\psi(\alpha_g^{-1}(a)))_{g \in G}(\beta_g^{-1}(\phi(1)^{1/2}))_{g \in G})\\
    &=& P((\phi(1)^{1/2})_{g \in G}(\psi(\alpha_g^{-1}(a)))_{g \in G}(\phi(1)^{1/2})_{g \in G})\\
    &=& P((\phi(1)^{1/2}\psi(\alpha_g^{-1}(a))\phi(1)^{1/2})_{g \in G})\\
    &=& P((\phi(\alpha_g^{-1}(a)))_{g \in G})\\
    &=& P((\beta_g^{-1}(\phi(a)))_{g \in G})\\
    &=& P(\iota_{\B}(\phi(a)))\\
    &=& \phi(a).
\end{eqnarray*}
\end{proof}

\section{Some General Results} \label{general_results}

\subsection{Finite-Dimensional Inclusions}

In this section we determine when an inclusion of finite-dimensional $C^*$-algebras has the various unique extension properties. This ought to have been done early in the development of the theory, but was somehow overlooked. Of course finite-dimensional $C^*$-algebras are von Neumann algebras, so \cite[Theorem 5.3]{PittsZarikian2015} gives a partial answer.

\begin{lemma} \label{fin_dim_summands}
Let $\A$ be a unital $C^*$-algebra and $p \in \Proj(\A)$. If $\dim(p\A p) < \infty$ and $\dim(p^\perp\A p^\perp) < \infty$, then $\dim(\A) < \infty$.
\end{lemma}

\begin{proof}
Define $X = p\A p^\perp$, $\C = \span(XX^*) \subseteq p\A p$, and $\D = \span(X^*X) \subseteq p^\perp\A p^\perp$ (no closures necessary because of finite dimensionality). Then $X$ is a equivalence $\C$-$\D$-bimodule \cite[Section 8.1.2]{BlecherLeMerdyBook}. Let $y_1, y_2, ..., y_n, z_1, z_2, ..., z_n \in X$ be such that $1_{\D}=\sum_{j=1}^n y_j^*z_j$. Note that $x1_{\D}=x$ for all $x \in X$, since
\[
    \|x1_{\D}-x\|^2 = \|(x1_{\D}-x)^*(x1_{\D}-x)\|
    = \|1_{\D}x^*x1_{\D}-1_{\D}x^*x-x^*x1_{\D}+x^*x\| = 0.
\]
By \cite[Proposition 8.1.11 and Lemma 8.1.15]{BlecherLeMerdyBook},
\[
    \bbK_{\D}(C_n(X)) \cong M_n(\bbK_{\D}(X)) \cong M_n(\C).
\]
We claim that the map $C_n(X) \to \bbK_{\D}(C_n(X)): \vec{x} \mapsto \theta_{\vec{x},\vec{y}}$ is a linear injection, where $\theta_{\vec{x},\vec{y}}(\vec{v}) = \vec{x} \cdot \langle \vec{y}, \vec{v} \rangle_{\D}$ for all $\vec{v} \in C_n(X)$. Indeed,
\[
    \theta_{\vec{x},\vec{y}}=0 \implies \theta_{\vec{x},\vec{y}}(\vec{z})=\vec{0} \implies \vec{x} \cdot \langle\vec{y},\vec{z}\rangle_{\D}=\vec{0} \implies \vec{x} \cdot 1_{\D}=\vec{0} \implies \vec{x}=\vec{0}.
\]
Thus
\[
    \dim(X) \leq \dim(C_n(X)) \leq \dim(\bbK_{\D}(C_n(X)))
    = \dim(M_n(\C)) < \infty.
\]
Therefore
\begin{eqnarray*}
    \dim(\A) 
    &=& \dim(p\A p)+\dim(p\A p^\perp)+\dim(p^\perp \A p)+\dim(p^\perp \A p^\perp)\\
    &=& \dim(p\A p)+\dim(X)+\dim(X^*)+\dim(p^\perp \A p^\perp) < \infty.
\end{eqnarray*}
\end{proof}

\begin{lemma}[{\cite[Lemma 6.6.3]{KR2}}] \label{rel_commutant_matrix_alg}
Let $\A \subseteq \B$ be a $C^*$-inclusion. If $\A \cong \bbM_n$, then $\B \cong \bbM_n \otimes (\A' \cap \B)$.
\end{lemma}

\begin{proof}
If $\{e_{ij}\} \subseteq \A$ is a self-adjoint system of matrix units, then the map
\[
    \B \to \bbM_n \otimes (\A' \cap \B): b \mapsto \begin{bmatrix} \sum_{k=1}^n e_{ki}be_{jk} \end{bmatrix}
\]
is a $*$-isomorphism.
\end{proof}

\begin{lemma} \label{fin_dim_commutant}
Let $\A \subseteq \B$ be a unital $C^*$-inclusion. If $\dim(\A) < \infty$ and $\dim(\B)=\infty$, then $\dim(\A' \cap \B)=\infty$.
\end{lemma}

\begin{proof}
Since $\dim(\A) < \infty$, there exist $z_1, z_2, ..., z_k \in \Proj(Z(\A))$ such that $z_1+z_2+...+z_k=\one$ and $\A z_i \cong \bbM_{n_i}$, $1 \leq i \leq k$. By Lemma \ref{fin_dim_summands}, there exists $1 \leq j \leq k$ such that $\dim(z_j\B z_j)=\infty$. Now by Lemma \ref{rel_commutant_matrix_alg},
\[
    \dim(z_j\B z_j) = n_j^2\dim((\A z_j)' \cap (z_j\B z_j)),
\]
which implies
\[
    \dim((\A z_j)' \cap (z_j\B z_j)) = \infty.
\]
Since
\[
    \A' \cap \B = \sum_{i=1}^k ((\A z_i)' \cap (z_i\B z_i)),
\]
we conclude that $\dim(\A' \cap \B)=\infty$.
\end{proof}

\begin{theorem} \label{finite_general}
Let $\A \subseteq \B$ be a unital $C^*$-inclusion. If $\dim(\A) < \infty$, then the following are equivalent:
\begin{enumerate}
\item[i.] $\A \subseteq \B$ has the PEP;
\item[ii.] $\A \subseteq \B$ has the AEP;
\item[iii.] $\A \subseteq \B$ is aperiodic;
\item[iv.] $\A \subseteq \B$ has a unique pseudo-expectation;
\item[v.] $\A \subseteq \B$ has a unique conditional expectation;
\item[vi.] $\A' \cap \B = Z(\A)$.
\end{enumerate}
In particular, if $\dim(\B)=\infty$, then all these conditions are false.
\end{theorem}

\begin{proof}
By Theorem \ref{hierarchy},
\[
    (i) \implies (ii) \implies (iii) \implies (iv) \implies (v),
\]
and \cite[Theorem 2.8]{BunceChu1998} gives $(v) \implies (i)$. Now $(i) \implies (vi)$ by \cite[Theorem 3.3]{BunceChu1998}, so it only remains to show that $(vi) \implies (i)$. Since
\[
    \dim(\A' \cap \B) = \dim(Z(\A)) \leq \dim(\A) < \infty,
\]
Lemma \ref{fin_dim_commutant} implies that $\dim(\B) < \infty$. Furthermore, the fact that $\dim(\A' \cap \B)=\dim(Z(\A))$ ensures that every simple summand of $\A$ embeds without multiplicity into a unique simple summand of $\B$. Now if $\phi \in PS(\A)$, then there exists $p \in \Proj(Z(\A))$ such that $\A p$ is simple, $\phi|_{\A p} \in PS(\A p)$, and $\phi|_{\A p^\perp}=0$. By the previous discussion, there exists a unique $q \in \Proj(Z(\B))$ such that $\B q$ is simple, $\A p \subseteq \B q$, and $\A p = p\B p$. If $\tilde{\phi} \in PS(\B)$ and $\tilde{\phi}|_{\A}=\phi$, then $\tilde{\phi}|_{\B q} \in PS(\B q)$ and $\tilde{\phi}|_{\B q^\perp}=0$. Since $\A p$ is a hereditary subalgebra of $\B q$, $\tilde{\phi}|_{\B q}$ is uniquely determined by $\phi|_{\A p}$ \cite[Theorem 5.1.13]{MurphyBook}. It follows that $\tilde{\phi}$ is uniquely determined by $\phi$.
\end{proof}

\subsection{Finite-Index Inclusions} \label{fin_ind}

Finite-index $C^*$-inclusions (in the sense of Watatani) substantially generalize inclusions of finite-dimensional $C^*$-algebras \cite[Proposition 2.4.2]{WatataniBook}. In this section, we characterize when a finite-index inclusion has a unique conditional expectation. Later in the paper, this result will be a key technical tool for proving that certain finite-index inclusions have a unique pseudo-expectation.

\begin{definition}[{\cite[Definition 1.2.2]{WatataniBook}}]
We say that a conditional expectation $E:\B \to \A$ has \emph{finite index} if there exists a set $Q = \{(u_1,v_1), (u_2,v_2), ..., (u_n,v_n)\} \subseteq \B \times \B$ such that
\[
    (\forall b \in \B) ~ b = \sum_j u_jE(v_jb) = \sum_j E(bu_j)v_j.
\]
In that case, we say that $Q$ is a \emph{quasi-basis} for $E$, and we define
\[
    \Ind(E) = \sum_j u_jv_j \in \B.
\]
We say that an inclusion $\A \subseteq \B$ has \emph{finite index} if it admits a finite-index conditional expectation.
\end{definition}

The following theorem collects some basic facts from the theory of finite-index inclusions.

\begin{theorem}[\cite{WatataniBook}] \label{fin_ind_facts}
Suppose $E:\B \to \A$ has finite index. Then the following statements hold:
\begin{enumerate}
\item[i.] Every quasi-basis for $E$ yields the same value of $\Ind(E)$. Furthermore, $\Ind(E) \in Z(\B)$. \cite[Proposition 1.2.8]{WatataniBook}
\item[ii.] There exists a quasi-basis for $E$ of the form $Q = \{(u_1,u_1^*), (u_2,u_2^*), ..., (u_n,u_n^*)\}$. Thus $\Ind(E) \geq 0$. \cite[Lemma 2.1.6]{WatataniBook}
\item[iii.] For all $b \in \B$,
\[
    E(b^*b) \geq \frac{1}{\|\Ind(E)\|}b^*b.
\]
In particular, $E$ is faithful. \cite[Proposition 2.6.2]{WatataniBook}
\end{enumerate}
\end{theorem}

Here is the main result of this section.

\begin{theorem} \label{unique_CE_fin_ind}
Let $\A \subseteq \B$ be a finite-index inclusion. Then $\A \subseteq \B$ has a unique conditional expectation if and only if $\A' \cap \B = Z(\A)$.
\end{theorem}

\begin{proof}
($\Rightarrow$) \cite[Proposition 3.1]{Zarikian2019a}

($\Leftarrow$) \cite[Corollary 1.4.3]{WatataniBook}
\end{proof}

\begin{remark}
Specializing this result to the case when $\dim(\B) < \infty$, we recover the equivalence $(5) \iff (6)$ in Theorem \ref{finite_general} above. Note, however, that Theorem \ref{finite_general} only assumes $\dim(\A) < \infty$.
\end{remark}

\subsection{Corners}

If $\A \subseteq \B$ is a $C^*$-inclusion and $0 \neq p \in \Proj(\A)$, then $p\A p \subseteq p\B p$ is the corresponding ``corner'' $C^*$-inclusion. It is easy to see that aperiodicity passes to corners, but it is unclear whether the unique pseudo-expectation property does so as well. This suggests that the Question \ref{open} above might actually have a negative answer. Indeed, if $\A \subseteq \B$ has a unique pseudo-expectation but $p\A p \subseteq p\B p$ doesn't, then $\A \subseteq \B$ is not aperiodic. At the end of this section, we prove that the unique pseudo-expectation property does pass to \ul{full} corners, which narrows the search for counterexamples considerably.

\begin{theorem} \label{corner_aperiodicity}
If $\A \subseteq \B$ is aperiodic and $0 \neq p \in \Proj(\A)$, then $p\A p \subseteq p\B p$ is aperiodic.
\end{theorem}

\begin{proof}
Let $b \in \B$, $\D \in \bbH(p\A p)$, and $\varepsilon > 0$. Then $\D \in \bbH(\A)$ and so there exist $d \in \D_+^1$ and $a \in \A$ such that $\|d(pbp)d-a\|<\varepsilon$. But then
\[
    \|d(pbp)d-pap\| = \|pd(pbp)dp-pap\| \leq \|d(pbp)d-a\| < \varepsilon.
\]
\end{proof}

\begin{lemma} \label{corner_iso}
Let $\A$ be a (not necessarily unital) $C^*$-algebra with multiplier algebra $M(\A)$. Suppose $v \in M(\A)$ is a partial isometry, and set $p=v^*v$ and $q=vv^*$. Then the map $\pi:p\A p \to q\A q: x \mapsto vxv^*$ is a $*$-isomorphism.
\end{lemma}

\begin{proof}
It is easy to see that $\pi$ is a $*$-homomorphism. Indeed,
\[
    \pi(xy) = vxyv^* = vxpyv^* = vxv^*vyv^* = \pi(x)\pi(y), \qquad x, y \in p\A p.
\]
By symmetry, $\sigma:q\A q \to p\A p: x \mapsto v^*xv$ is a $*$-homomorphism. Now
\[
    \sigma(\pi(x)) = \sigma(vxv^*) = v^*vxv^*v = pxp = x, \qquad x \in p\A p.
\]
Again by symmetry, $\sigma=\pi^{-1}$.
\end{proof}

The following result is joint work with David Pitts, who kindly agreed to its inclusion in this paper.

\begin{theorem}[Pitts-Zarikian] \label{PZ}
Let $\A \subseteq \B$ be a unital $C^*$-inclusion. If $p \in \Proj(\A)$ is \ul{full} (not contained in a proper ideal), then there exists a bijective correspondence between the pseudo-expectations for $\A \subseteq \B$ and the pseudo-expectations for $p\A p \subseteq p\B p$. In particular, $p\A p \subseteq p\B p$ has a unique pseudo-expectation if and only if $\A \subseteq \B$ has a unique pseudo-expectation.
\end{theorem}

\begin{proof}
Denote by $\bbK$ the compact operators on $\ell^2$. By \cite[Lemma 2.5]{Brown1977}, there exists a partial isometry $v \in M(\bbK \otimes \A)$ such that $v^*v=I \otimes 1$ and $vv^*=I \otimes p$. Now $M(\bbK \otimes \A) \subseteq M(\bbK \otimes \B)$ \cite[Section 3.12.12]{PedersenBook}, which implies $v \in M(\bbK \otimes \B)$. Thus by Lemma \ref{corner_iso}, $\bbK \otimes \B \cong \bbK \otimes (p\B p)$ via a $*$-isomorphism that maps $\bbK \otimes \A$ onto $\bbK \otimes (p\A p)$. It follows that the pseudo-expectations for $\bbK \otimes \A \subseteq \bbK \otimes \B$ and $\bbK \otimes (p\A p) \subseteq \bbK \otimes (p\B p)$ are in bijective correspondence. By \cite[Theorem 2.2.1]{Zarikian2025}, the pseudo-expectations for $\A \subseteq \B$ and $\bbK \otimes \A \subseteq \bbK \otimes \B$ are in bijective correspondence. Likewise, the pseudo-expectations for $\bbK \otimes (p\A p) \subseteq \bbK \otimes (p\B p)$ and $p\A p \subseteq p\B p$ are in bijective correspondence. Thus the pseudo-expectations for $\A \subseteq \B$ and $p\A p \subseteq p\B p$ are in bijective correspondence.
\end{proof}

We leave the general question as an open problem.

\begin{question} \label{open_corner}
If $\A \subseteq \B$ has a unique pseudo-expectation and $0 \neq p \in \Proj(\A)$ is arbitrary, does $p\A p \subseteq p\B p$ have a unique pseudo-expectation?
\end{question}

\section{Dynamical Results} \label{dynamical_results}

\subsection{The Inclusion $C_r^*(G) \subseteq \A \rtimes_{\alpha,r} G$}

\phantom{}\\

In this section we analyze unique extension properties for the $C^*$-inclusion $C_r^*(G) \subseteq \A \rtimes_{\alpha,r} G$, showing that they are related to ``unique ergodicity'' of the action. For abelian groups and finite groups, we obtain complete characterizations. For general discrete groups, we obtain distinct necessary and sufficient conditions, neither of which seems easy to check because they involve equivariant ucp maps $\A \to C(\partial_FG)$, where $\partial_FG$ is Furstenberg boundary of $G$ (see Section \ref{inj} above). The situation is more tractable if one specializes to amenable groups, since $C(\partial_FG)=\bbC$ in that case.

\begin{lemma}[{\cite[Exercise 4.1.4]{BrownOzawaBook}}] \label{equivariant_extension}
Let $(\A,G,\alpha)$, $(\B,G,\beta)$ be $C^*$-dynamical systems and $\phi:\A \to \B$ be a $G$-equivariant ucp map. Then there exists a unique ucp map $\Phi:\A \rtimes_{\alpha,r} G \to \B \rtimes_{\beta,r} G$ such that
\[
    (\forall a \in C_c(G,\A)) \quad \Phi\left(\sum_g a_g\lambda_g\right) = \sum_g \phi(a_g)\lambda_g.
\]
(Some authors write $\Phi = \phi \rtimes G$.) In particular, if $\phi:\A \to \bbC$ is a $G$-invariant state, then $\Phi:\A \rtimes_{\alpha,r} G \to C_r^*(G)$ is a conditional expectation. Lastly, $\Phi$ is faithful if and only if $\phi$ is faithful.
\end{lemma}

\begin{lemma} \label{rel_commutant_condition}
Let $(\A,G,\alpha)$ be a $C^*$-dynamical system and $x \sim \sum_g x_g\lambda_g \in \A \rtimes_{\alpha,r} G$. Then
\[
    x \in C_r^*(G)' \cap (\A \rtimes_{\alpha,r} G)
    \iff \alpha_h(x_g)=x_{hgh^{-1}} \text{ for all } g, h \in G.
\]
\end{lemma}

\begin{proof}
\begin{eqnarray*}
    x \in C_r^*(G)' \cap (\A \rtimes_{\alpha,r} G)
    &\iff& \lambda_hx=x\lambda_h \text{ for all } h \in G\\
    &\iff& \lambda_hx\lambda_h^*=x \text{ for all } h \in G\\
    &\iff& \sum_g \alpha_h(x_g)\lambda_{hgh^{-1}}=\sum_g x_g\lambda_g \text{ for all } h \in G\\
    &\iff& \alpha_h(x_g)=x_{hgh^{-1}} \text{ for all } g, h \in G.
\end{eqnarray*}
\end{proof}

\subsubsection{The case of $G$ abelian}

We denote by $S_G(\A)$ the $G$-invariant states on $\A$.

\begin{theorem} \label{abelian_main}
Let $(\A,G,\alpha)$ be a $C^*$-dynamical system. If $G$ is \ul{abelian}, then the following are equivalent:
\begin{enumerate}
\item[i.] $C_r^*(G) \subseteq \A \rtimes_{\alpha,r} G$ has the PEP;
\item[ii.] $C_r^*(G) \subseteq \A \rtimes_{\alpha,r} G$ has the AEP;
\item[iii.] $C_r^*(G) \subseteq \A \rtimes_{\alpha,r} G$ is aperiodic;
\item[iv.] $C_r^*(G) \subseteq \A \rtimes_{\alpha,r} G$ has a unique pseudo-expectation;
\item[v.] $C_r^*(G) \subseteq \A \rtimes_{\alpha,r} G$ has a unique conditional expectation;
\item[vi.] $|S_G(\A)|=1$ (i.e., $\alpha:G \curvearrowright \A$ is uniquely ergodic).
\end{enumerate}
If these conditions hold, then the following equivalent conditions also hold:
\begin{enumerate}
\item[vii.] $C_r^*(G) \subseteq \A \rtimes_{\alpha,r} G$ is a MASA;
\item[viii.] $\A^G=\bbC\one$.
\end{enumerate}
Schematically,
\[
    (i) \iff (ii) \iff (iii) \iff (iv) \iff (v) \iff (vi) \implies (vii) \iff (viii).
\]
\end{theorem}

\begin{proof}
By Theorem \ref{hierarchy} and Lemma \ref{equivariant_extension},
\[
    (i) \implies (ii) \implies (iii) \implies (iv) \implies (v) \implies (vi),
\]
and by \cite[Theorem 3.3]{BunceChu1998}, $(i) \implies (vii)$.\\

($(vi) \implies (i)$) Suppose $\omega \in S_G(\A)$ is the unique $G$-invariant state on $\A$. Fix $\phi \in PS(C_r^*(G))$ and let $\tilde{\phi} \in S(\A \rtimes_{\alpha,r} G)$ be such that $\tilde{\phi}|_{C_r^*(G)}=\phi$. Since $C_r^*(G)$ is abelian, $\phi$ is multiplicative. Thus for any $a \in C_c(G,\A)$, we have that
\[
    \tilde{\phi}\left(\sum_g a_g\lambda_g\right)=\sum_g \tilde{\phi}(a_g\lambda_g) = \sum_g \tilde{\phi}(a_g)\phi(\lambda_g).
\]
Now for any $a \in \A$ and $g \in G$,
\begin{eqnarray*}
    \tilde{\phi}(\alpha_g(a)) 
    &=& \tilde{\phi}(\lambda_ga\lambda_g^*) = \phi(\lambda_g)\tilde{\phi}(a)\phi(\lambda_g^*)\\
    &=& \tilde{\phi}(a)\phi(\lambda_g)\phi(\lambda_g^*) = \tilde{\phi}(a)\phi(\lambda_g\lambda_g^*)\\
    &=& \tilde{\phi}(a)\phi(\one) = \tilde{\phi}(a). 
\end{eqnarray*}
Thus $\tilde{\phi}|_{\A} \in S_G(\A)$. It follows that $\tilde{\phi}|_{\A}=\omega$, which implies
\[
    (\forall a \in C_c(G,\A)) \quad \tilde{\phi}\left(\sum_g a_g\lambda_g\right)=\sum_g \omega(a_g)\phi(\lambda_g).
\]
($(vii) \implies (viii)$) Suppose $C_r^*(G) \subseteq \A \rtimes_{\alpha,r} G$ is a MASA. Let $a \in \A^G$. Then
\[
    \lambda_ga = \alpha_g(a)\lambda_g = a\lambda_g, ~ g \in G,
\]
which implies
\[
    a \in C_r^*(G)' \cap (\A \rtimes_{\alpha,r} G) = C_r^*(G).
\]
Thus
\[
    a \in \A \cap C_r^*(G) = \bbC\one.
\]
($(viii) \implies (vii)$) Conversely, suppose $\A^G=\bbC\one$. Let
\[
    x \sim \sum_g x_g\lambda_g \in C_r^*(G)' \cap (\A \rtimes_{\alpha,r} G).
\]
By Lemma \ref{rel_commutant_condition},
\[
    \alpha_h(x_g)=x_{hgh^{-1}}=x_g \text{ for all } g, h \in G.
\]
Thus
\[
    x_g \in \A^G = \bbC\one \text{ for all } g \in G.
\]
Since $G$ has the Approximation Property, $x \in C_r^*(G)$ \cite[Proposition 3.4]{Suzuki2017}.
\end{proof}

\begin{example}
Keeping the notation of Theorem \ref{abelian_main}, we show that the implication $(viii) \implies (vi)$ is not true in general.
\end{example}

\begin{proof}
By \cite[Corollary 1]{Elek2021}, there exists a free minimal action $\sigma:\bbZ \curvearrowright X$ on the Cantor set which is not uniquely ergodic. Let $(C(X),\bbZ,\alpha)$ be the corresponding $C^*$-dynamical system. Fix $f \in C(X)$ and suppose $\alpha_k(f)=f$ for all $k \in \bbZ$. Let $x_0 \in X$. Then for any $k \in \bbZ$, $f(\sigma_k(x_0))=f(x_0)$. Since $\sigma:\bbZ \curvearrowright X$ is minimal, $\{\sigma_k(x_0): k \in \bbZ\}$ is dense in $X$, which implies $f=f(x_0)\one$. Thus $C(X)^{\bbZ}=\bbC\one$. But $|S_{\bbZ}(C(X))| \neq 1$, since $\sigma:\bbZ \curvearrowright X$ is not uniquely ergodic.
\end{proof}

\subsubsection{The case of $G$ finite} If $G$ is a group and $g \in G$, then
\[
    C_G(g) = \{h \in G: hg=gh\}
\]
is the \emph{centralizer} of $g$, which is a subgroup of $G$. Of course if $G$ is abelian, then $C_G(g)=G$ for all $g \in G$.

\begin{theorem} \label{finite_main}
Let $(\A,G,\alpha)$ be a $C^*$-dynamical system. If $|G| < \infty$, then the following are equivalent:
\begin{enumerate}
\item[i.] $C_r^*(G) \subseteq \A \rtimes_{\alpha,r} G$ has the PEP;
\item[ii.] $C_r^*(G) \subseteq \A \rtimes_{\alpha,r} G$ has the AEP;
\item[iii.] $C_r^*(G) \subseteq \A \rtimes_{\alpha,r} G$ is aperiodic;
\item[iv.] $C_r^*(G) \subseteq \A \rtimes_{\alpha,r} G$ has a unique pseudo-expectation;
\item[v.] $C_r^*(G) \subseteq \A \rtimes_{\alpha,r} G$ has a unique conditional expectation;
\item[vi.] $|S_{C_G(g)}(\A)|=1$ for all $g \in G$ (i.e., $\alpha:C_G(g) \curvearrowright \A$ is uniquely ergodic for all $g \in G$);
\item[vii.] $C_r^*(G)' \cap (\A \rtimes_{\alpha,r} G) = Z(C_r^*(G))$;
\item[viii.] $\A^{C_G(g)}=\bbC\one$ for all $g \in G$.
\end{enumerate}
In particular, if $\dim(\A)=\infty$, then all of these conditions are false.
\end{theorem}

\begin{proof}
By Theorem \ref{finite_main},
\[
    (i) \iff (ii) \iff (iii) \iff (iv) \iff (v) \iff (vii).
\]
($(vi) \iff (viii)$) For any subgroup $H \subseteq G$, there is a faithful conditional expectation $M_H:\A \to \A^H$ defined by
\[
    M_H(a)=\frac{1}{|H|}\sum_{h \in H} \alpha_h(a), ~ a \in \A.
\]
Then the map
\[
    S_H(\A) \to S(\A^H): \omega \mapsto \omega|_{\A^H}
\]
is a bijection with inverse
\[
    S(\A^H) \to S_H(\A): \phi \mapsto \phi \circ M_H.
\]
Thus
\[
    |S_H(\A)|=1 \iff |S(\A^H)|=1 \iff \A^H=\bbC\one.
\]
($(vii) \implies (viii)$) Let $g \in G$ and $a \in \A^{C_G(g)}$. Define
\[
    x = \sum_{h \in G} \alpha_h(a)\lambda_{hgh^{-1}} \in \A \rtimes_{\alpha,r} G
\]
and note that $x_g=|C_G(g)|a$. For all $k \in G$,
\[
    \lambda_kx = \sum_{h \in G} \alpha_{kh}(a)\lambda_{(kh)g(kh)^{-1}}\lambda_k = x\lambda_k.
\]
Thus
\[
    x \in C_r^*(G)' \cap (\A \rtimes_{\alpha,r} G) = Z(C_r^*(G)) \implies x_g \in \bbC\one
    \implies a \in \bbC\one.
\]
($(viii) \implies (vii)$) Let
\[
    x = \sum_{g \in G} x_g\lambda_g \in C_r^*(G)' \cap (\A \rtimes_{\alpha,r} G).
\]
Fix $g \in G$. By Lemma \ref{rel_commutant_condition},
\[
    \alpha_h(x_g)=x_{hgh^{-1}} \text{ for all } h \in G.
\]
Thus 
\[
    \alpha_h(x_g)=x_g \text{ for all } h \in C_G(g).
\]
That is, $x_g \in \A^{C_G(g)}=\bbC\one$. Since $g$ was arbitrary, $x \in C_r^*(G)$, which implies $x \in Z(C_r^*(G))$.
\end{proof}

\begin{example}
Suppose $G = S_3$ acts on $\A = \bbC \oplus \bbC \oplus \bbC$ by permuting the summands. Then $C_r^*(G) \subseteq \A \rtimes_r G$ does not have a unique conditional expectation, even though $G \curvearrowright \A$ is uniquely ergodic.
\end{example}

\begin{proof}
Clearly $a \oplus b \oplus c \mapsto \frac{a+b+c}{3}$ is the unique $G$-invariant state on $\A$. Now let
\[
    H = C_G((12)) = \{(),(12)\}.
\]
Then
\[
    \A^H = \{a \oplus a \oplus b: a, b \in \bbC\} \neq \bbC\one.
\]
By Theorem \ref{finite_main}, $C_r^*(G) \subseteq \A \rtimes_r G$ does not have a unique conditional expectation. In fact, the inclusion $C_r^*(G) \subseteq \A \rtimes_r G$ is isomorphic to the inclusion
\[
    \bbC \oplus \bbM_2 \oplus \bbC \to \bbM_3 \oplus \bbM_3: a \oplus B \oplus c \mapsto (a \oplus B) \oplus (B \oplus c).
\]
There are clearly multiple conditional expectations $\bbM_3 \oplus \bbM_3 \to \bbC \oplus \bbM_2 \oplus \bbC$, for example
\[
    E_1(V \oplus W) = v_{11} \oplus \begin{bmatrix} v_{22} & v_{23}\\ v_{32} & v_{33} \end{bmatrix} \oplus w_{33}
\]
and
\[
    E_2(V \oplus W) = v_{11} \oplus \begin{bmatrix} w_{11} & w_{12}\\ w_{21} & w_{22} \end{bmatrix} \oplus w_{33}.
\]
\end{proof}

\subsubsection{The general case}

\phantom{}\\

Recall from Section \ref{inj} that $I_G(\bbC)=C(\partial_FG)$, where $\partial_FG$ is the Furstenberg boundary of $G$. Thus for any $C^*$-dynamical system $(\A,G,\alpha)$, there exists at least one $G$-equivariant ucp map $\phi:\A \to C(\partial_FG)$.

\begin{theorem} \label{main1}
Let $(\A,G,\alpha)$ be a $C^*$-dynamical system. Consider the statements:
\begin{enumerate}
\item[i.] For each $g \in G$, there exists a unique $C_G(g)$-equivariant ucp map $\phi_g:\A \to C(\partial_FG)$;
\item[ii.] $C_r^*(G) \subseteq \A \rtimes_{\alpha,r} G$ has a unique pseudo-expectation;
\item[iii.] There exists a unique $G$-equivariant ucp map $\phi:\A \to C(\partial_FG)$.
\end{enumerate}
Then
\[
    (i) \implies (ii) \implies (iii).
\]
\end{theorem}

\begin{proof}
((i) $\implies$ (ii)) As discussed in Section \ref{inj},
\[
    C(\partial_FG) \rtimes_{\sigma,r} G \subseteq I(C_r^*(G)) \subseteq C(\partial_FG) \overline{\rtimes}_\sigma G,
\]
where $\sigma:G \curvearrowright C(\partial_FG)$ is the canonical action. Suppose $\Psi:\A \rtimes_{\alpha,r} G \to C(\partial_FG) \overline{\rtimes}_\sigma G$ is a ucp map such that $\Psi|_{C_r^*(G)}=\id$. Then for all $a \in \A$,
\[
    \Psi(a) \sim \sum_g \psi_g(a)\lambda_g,
\]
where $\psi_g:\A \to C(\partial_FG)$ is defined by
\[
    \psi_g(a) = \overline{\bbE}(\Psi(a)\lambda_g^*), ~ a \in \A, 
\]
and $\overline{\bbE}:C(\partial_FG) \overline{\rtimes}_\sigma G \to C(\partial_FG)$ is the canonical faithful $G$-equivariant conditional expectation. Fix $g \in G$ and let $h \in C_G(g)$. Then
\begin{eqnarray*}
    \psi_g(\alpha_h(a))
    &=& \overline{\bbE}(\Psi(\alpha_h(a))\lambda_g^*)\\
    &=& \overline{\bbE}(\Psi(\lambda_ha\lambda_h^*)\lambda_g^*)\\
    &=& \overline{\bbE}(\lambda_h\Psi(a)\lambda_h^*\lambda_g^*)\\
    &=& \overline{\bbE}(\lambda_h\Psi(a)\lambda_g^*\lambda_h^*)\\
    &=& \sigma_h\left(\overline{\bbE}(\Psi(a)\lambda_g^*)\right)\\
    &=& \sigma_h(\psi_g(a)).
\end{eqnarray*}
That is, $\psi_g:\A \to C(\partial_FG)$ is $C_G(g)$-equivariant. Now $C(\partial_FG)$ is $C_G(g)$-injective by \cite[Lemma 7.5]{BussEchterhoffWillett2020}. Thus Lemmas \ref{equiv_cb_decomp} and \ref{equiv_cp_decomp} imply that $\psi_g=\psi_g(1)\phi_g$. Since
\[
    \psi_g(1) = \overline{\bbE}(\Psi(1)\lambda_g^*) = \overline{\bbE}(\lambda_g^*) = \begin{cases} 1, & g=e\\ 0, & g \neq e \end{cases},
\]
we have that $\Psi(a)=\phi_e(a)$ for all $a \in \A$. Therefore
\[
    \Psi\left(\sum_g a_g\lambda_g\right) = \sum_g \Psi(a_g)\lambda_g = \sum_g \phi_e(a_g)\lambda_g, ~ a \in C_c(G,\A).
\]

($(ii) \implies (iii)$) Let $\phi:\A \to C(\partial_FG)$ be a $G$-equivariant ucp map. By Lemma \ref{equivariant_extension}, there exists a ucp map
\[
    \Phi:\A \rtimes_{\alpha,r} G \to C(\partial_FG) \rtimes_{\sigma,r} G
\]
such that
\[
    \Phi\left(\sum_g a_g\lambda_g\right) = \sum_g \phi(a_g)\lambda_g, ~ a \in C_c(G,\A).
\]
Since $C(\partial_FG) \rtimes_{\sigma,r} G \subseteq I(C_r^*(G))$, we see that $\Phi$ is a pseudo-expectation for $C_r^*(G) \subseteq \A \rtimes_{\alpha,r} G$. Because $\Phi$ is uniquely determined, so is $\phi$. 
\end{proof}

\begin{remark} \label{too_strong}
In general, condition (i) of Theorem \ref{main1} is much too strong. Indeed, as explained in Section \ref{inj},
\[
    C_r^*(G) \subseteq C(\partial_FG) \rtimes_{\sigma,r} G \subseteq I(C_r^*(G)).
\]
Since $C_r^*(G) \subseteq I(C_r^*(G))$ has a unique pseudo-expectation, so does $C_r^*(G) \subseteq C(\partial_FG) \rtimes_{\sigma,r} G$. In other words, $C_r^*(G) \subseteq C(\partial_FG) \rtimes_{\sigma,r} G$ satisfies condition (ii) of Theorem \ref{main1}. Now assume that $C_r^*(G) \subseteq C(\partial_FG) \rtimes_{\sigma,r} G$ also satisfies condition (i) of Theorem \ref{main1}. Then for each $g \in G$, $\id:C(\partial_FG) \to C(\partial_FG)$ is the only $C_G(g)$-equivariant ucp map, which implies $I_{C_G(g)}(\bbC)=C(\partial_FG)$. But if $G=\bbF_2$ (the free group on two generators), then $C(\partial_FG) \neq \bbC$ since $G$ is not amenable, while $I_{C_G(g)}(\bbC)=\bbC$ for all $e \neq g \in G$ since $C_G(g) \cong \bbZ$.
\end{remark}

\begin{corollary} \label{trivial_fixed}
Let $(\A,G,\alpha)$ be a $C^*$-dynamical system. If $C_r^*(G) \subseteq \A \rtimes_{\alpha,r} G$ has a unique pseudo-expectation, then $\A^G=\bbC1$. 
\end{corollary}

\begin{proof}
Let $\phi \in S(\A^G)$. Since $\phi$ is (trivially) $G$-equivariant, there exists a $G$-equivariant ucp map $\tilde{\phi}:\A \to C(\partial_FG)$ such that $\tilde{\phi}|_{\A^G}=\phi$. By Theorem \ref{main1}, $\tilde{\phi}$ is uniquely determined. Thus $\phi$ is uniquely determined, which implies $\A^G=\bbC1$.
\end{proof}

If $G$ is amenable, then $C(\partial_FG)=\bbC$. Thus Theorem \ref{main1} has the following immediate consequence.

\begin{corollary}
Let $(\A,G,\alpha)$ be a $C^*$-dynamical system, with $G$ \ul{amenable}. Consider the statements:
\begin{enumerate}
\item[i.] For each $g \in G$, the restricted action $C_g(G) \curvearrowright \A$ is uniquely ergodic;
\item[ii.] $C_r^*(G) \subseteq \A \rtimes_{\alpha,r} G$ has a unique pseudo-expectation;
\item[iii.] The full action $G \curvearrowright \A$ is uniquely ergodic.
\end{enumerate}
Then
\[
    (i) \implies (ii) \implies (iii).
\]
\end{corollary}

\begin{remark} \label{C*-simple}
Among non-amenable groups, it would be particularly interesting to see how Theorem \ref{main1} specializes to $C^*$-simple groups, i.e., those groups $G$ for which $C_r^*(G)$ is simple ($\bbF_2$, for example). As noted in Remark \ref{too_strong} above, condition (i) of Theorem \ref{main1} is much too strong in that case. Thanks to \cite{KalantarKennedy2017} and \cite{Kennedy2020}, the reader who wishes to pursue this direction would have several reformulations of $C^*$-simplicity to choose from.
\end{remark}

\subsection{The Inclusion $\A^G \subseteq \A$}

\phantom{}\\

Let $(\A,G,\alpha)$ be a $C^*$-dynamical system. In this section we analyze the unique pseudo-expectation property for the $C^*$-inclusion $\A^G \subseteq \A$. In order to make the problem more tractable, we restrict our attention to the situation $|G| < \infty$. In that case, there is a faithful conditional expectation $M_G:\A \to \A^G$ given by
\[
    M_G(a) = \frac{1}{|G|}\sum_g \alpha_g(a), ~ a \in \A,
\]
which will be the unique pseudo-expectation when $\A^G \subseteq \A$ has the unique pseudo-expectation property.\\

\subsubsection{Preliminaries}

A key technical tool in our analysis is the fact that $\A^G \subseteq \A$ is isomorphic to a corner of $\A \rtimes_{\alpha,r} G \subseteq \A \otimes B(\ell^2(G))$, and that the latter $C^*$-inclusion always has finite Watatani index. We explain both of these statements presently.\\

Recall from Section \ref{notation} that there exists an embedding $\iota:\A \rtimes_{\alpha,r} G \to \A \otimes B(\ell^2(G))$ such that
\[
    \iota(a) = \sum_g (\alpha_g^{-1}(a) \otimes E_{g,g}), ~ a \in \A,
\]
and
\[
    \iota(\lambda_g) = 1 \otimes \lambda_g = \sum_h (1 \otimes E_{gh,h}), ~ g \in G.
\]
In particular, $\iota(\A) \neq \A \otimes I$ unless $\alpha:G \curvearrowright \A$ is trivial. In fact, $\iota(\A) \cap (\A \otimes I) = \A^G \otimes I = \iota(\A^G)$. To simplify notation, we will typically identify $\A \rtimes_{\alpha,r} G$ with $\iota(\A \rtimes_{\alpha,r} G) \subseteq \A \otimes B(\ell^2(G))$. 

\begin{lemma} \label{fin_ind_CE}
Let $\alpha:G \curvearrowright \A$ be the action of a finite group on a unital $C^*$-algebra. Then there is a conditional expectation $\E:\A \otimes B(\ell^2(G)) \to \A \rtimes_{\alpha,r} G$ such that $\Ind(\E)=|G|(1 \otimes I)$.
\end{lemma}

\begin{proof}
Define a linear map $\E:\A \otimes B(\ell^2(G)) \to \A \rtimes_{\alpha,r} G$ by
\[
    \E\left(\begin{bmatrix} x_{g,h} \end{bmatrix}\right) = \sum_g\left(\frac{1}{|G|}\sum_h \alpha_h(x_{h,g^{-1}h})\right)\lambda_g.
\]
It is easy to check that for all $a \in \A$ and $g, h \in G$,
\[
    \E(a \otimes E_{g,h}) = \frac{1}{|G|}\alpha_g(a)\lambda_{gh^{-1}}.
\]
Thus
\begin{eqnarray*}
    \E(a\lambda_g)
    &=& \E\left(\sum_h (\alpha_h^{-1}(a) \otimes E_{h,g^{-1}h})\right) = \sum_h \E(\alpha_h^{-1}(a) \otimes E_{h,g^{-1}h})\\
    &=& \sum_h \frac{1}{|G|}\alpha_h(\alpha_h^{-1}(a))\lambda_{h(g^{-1}h)^{-1}} = \sum_h \frac{1}{|G|}a\lambda_g = a\lambda_g,
\end{eqnarray*}
which shows that $\E|_{\A \rtimes_{\alpha,r} G}=\id$. To show that $\E$ is positive, it suffices to observe that
\[
    \E\left(\begin{bmatrix} a_g^*a_h \end{bmatrix}\right)
    = \sum_g\left(\frac{1}{|G|}\sum_h \alpha_h(a_h^*a_{g^{-1}h})\right)\lambda_g
    = \frac{1}{|G|}\left(\sum_g a_g\lambda_g^*\right)^*\left(\sum_h a_h\lambda_h^*\right) \geq 0
\]
\cite[Lemma 3.13]{PaulsenBook}. Then $\E$ is a contraction \cite[Corollary 2.9]{PaulsenBook}, therefore a conditional expectation \cite[Theorem II.6.10.2]{BlackadarBook}. Finally, set $U_g = |G|^{1/2}(1 \otimes E_{g,e}) \in \A \otimes B(\ell^2(G))$ for all $g \in G$. We claim that $\{(U_g,U_g^*): g \in G\}$ is a quasi-basis for $\E$. Indeed, 
\begin{eqnarray*}
    \sum_k \E((a \otimes E_{g,h})U_k)U_k^*
    &=& |G|\sum_k \E((a \otimes E_{g,h})(1 \otimes E_{k,e}))(1 \otimes E_{e,k})\\
    &=& |G|\E(a \otimes E_{g,e})(1 \otimes E_{e,h}) = \alpha_g(a)\lambda_g(1 \otimes E_{e,h})\\
    &=& \sum_k (\alpha_k^{-1}(\alpha_g(a)) \otimes E_{k,g^{-1}k})(1 \otimes E_{e,h}) = a \otimes E_{g,h}.
\end{eqnarray*}
Thus
\[
    \Ind(\E) = \sum_g U_gU_g^* = |G|\sum_g (1 \otimes E_{g,g}) = |G|(1 \otimes I).
\] 
\end{proof}

\begin{lemma}[\cite{Rosenberg1979}] \label{corner_iso}
Let $\alpha:G \curvearrowright \A$ be an action of a finite group on a unital $C^*$-algebra and set $p = \frac{1}{|G|}\sum_g \lambda_g \in \Proj(\A \rtimes_{\alpha,r} G)$. Then the following statements hold:
\begin{enumerate}
\item[i.] $p = 1 \otimes P$, where
\[
    P = \frac{1}{|G|}\begin{bmatrix} 1 & 1 & ... & 1\\ 1 & 1 & ... & 1\\ \vdots & \vdots & \ddots & \vdots\\ 1 & 1 & ... & 1 \end{bmatrix} \in \Proj(\B(\ell^2(G))).
\]
\item[ii.] The map $\pi:\A \to p(\A \otimes B(\ell^2(G)))p:a \mapsto a \otimes P$ is a $*$-isomorphism such that $\pi(\A^G) = p(\A \rtimes_{\alpha,r} G)p$.
\item[iii.] For all $a \in \A$, $\pi(M_G(a)) = \E(\pi(a))$, where $M_G:\A \to \A^G$ and $\E:\A \otimes B(\ell^2(G)) \to \A \rtimes_{\alpha,r} G$ are the canonical conditional expectations.
\end{enumerate}
\end{lemma}

\begin{proof}
Simple calculations show that $p = 1 \otimes P$ and that $\pi$ is a $*$-isomorphism. To see that $\pi(\A^G)=p(\A \rtimes_{\alpha,r} G)p$, note that
\[
    p\left(\sum_g a_g\lambda_g\right)p = M_G\left(\sum_g a_g\right)p = M_G\left(\sum_g a_g\right) \otimes P.
\]
Lastly,
\[
    \E(a \otimes P) = \frac{1}{|G|}\sum_g\sum_h \E(a \otimes E_{g,h})
    = \frac{1}{|G|}\sum_g\left(\frac{1}{|G|}\sum_h \alpha_g(a)\lambda_{gh^{-1}}\right)
    = M_G(a) \otimes P,
\]
which shows that $\pi \circ M_G = \E \circ \pi$.
\end{proof}

If $p$ was always a full projection, then studying the pseudo-expectations of $\A^G \subseteq \A$ would be equivalent to studying the pseudo-expectations of $\A \rtimes_{\alpha,r} G \subseteq \A \otimes B(\ell^2(G))$ (see Theorem \ref{PZ} above). But the situation is more complicated than that, so additional care will be needed.

\begin{definition}[{\cite[Definition 7.1.4 and Lemma 7.1.9]{PhillipsBook}}]
We say that an action $\alpha:G \curvearrowright \A$ of a finite group on a unital $C^*$-algebra is \emph{saturated} if there exist $u_1, u_2, ..., u_n, v_1, v_2, ..., v_n \in \A$ such that
\[
    \sum_j u_j\alpha_g(v_j) = \begin{cases} 1, & g=e\\ 0, & g \neq e \end{cases}.
\]
\end{definition}

\begin{theorem}[\cite{JeongPark2008}] \label{saturated}
Let $\alpha:G \curvearrowright \A$ be an action of a finite group on a unital $C^*$-algebra and $p = \frac{1}{|G|}\sum_g \lambda_g \in \Proj(\A \rtimes_\alpha G)$. Then the following are equivalent:
\begin{enumerate}
\item[i.] $p$ is full;
\item[ii.] $\alpha$ is saturated;
\item[iii.] $\Ind(M_G)=|G|1$, where $M_G:\A \to \A^G$ is the canonical conditional expectation.
\end{enumerate}
\end{theorem}

\begin{proof}
The equivalence (i $\iff$ ii) is well-known. In fact, some authors use the fullness of $p$ as the definition of the saturation of $\alpha$. The reader who wants to work out the details should start with the fact that $p$ is full if and only if there exist $x_1, x_2, ..., x_n, y_1, y_2, ..., y_n \in \A \rtimes_{\alpha,r} G$ such that $\sum_j x_jpy_j=1$ \cite[Exercise 4.8]{RordamLarsenLaustenBook}. The equivalence (ii $\iff$ iii) is more recent \cite[Theorem 4.1]{JeongPark2008}, although with some effort one can give a proof which does not involve Hopf $*$-algebras. 
\end{proof}

\begin{remark}
If $\alpha:G \curvearrowright \A$ is not saturated, then it can happen that $M_G:\A \to \A^G$ does not have finite index \cite[Proposition 2.8.2]{WatataniBook}, in which case $\A^G \subseteq \A$ is not a finite-index inclusion, since $M_G(a^*a) \geq \frac{1}{|G|}a^*a$ for all $a \in \A$ \cite[Proposition 2.10.10]{WatataniBook}.
\end{remark}

Combining Lemma \ref{fin_ind_CE} and Theorem \ref{unique_CE_fin_ind}, we obtain the following important observation, which we will use repeatedly.

\begin{lemma} \label{unique_CE_crossed}
Let $\alpha:G \curvearrowright \A$ be an action of a finite group on a unital $C^*$-algebra. Then $\A \rtimes_{\alpha,r} G \subseteq \A \otimes B(\ell^2(G))$ has a unique conditional expectation if and only if
\[
    (\A \rtimes_{\alpha,r} G)' \cap (\A \otimes B(\ell^2(G))) = Z(\A \rtimes_{\alpha,r} G).
\]
\end{lemma}

The following result will sometimes enable us to determine when Lemma \ref{unique_CE_crossed} applies.

\begin{lemma} \label{rel_comm_cross_prod}
Let $\alpha:G \curvearrowright \A$ be an action of a finite group on a unital $C^*$-algebra. Then
\[
    (\A \rtimes_\alpha G)' \cap (\A \otimes B(\ell^2(G))
    = \left\{\sum_g (d_g \otimes \rho_g): \begin{tabular}{c} $(\forall g \in G$) $d_g \in \A$ is a dependent\\ element for $\alpha_g \in \Aut(\A)$ \end{tabular}\right\},
\]
where $\rho:G \to B(\ell^2(G))$ is the right regular representation.
\end{lemma}

\begin{proof}
Straightforward calculations show that 
\[
    \iota(a)\begin{bmatrix} x_{g,h} \end{bmatrix} = \begin{bmatrix} x_{g,h} \end{bmatrix}\iota(a)
    \iff (\forall g, h \in G) ~ \alpha_g^{-1}(a)x_{g,h} = x_{g,h}\alpha_h^{-1}(a)
\]
and
\[
    \iota(\lambda_k)\begin{bmatrix} x_{g,h} \end{bmatrix}
    = \begin{bmatrix} x_{g,h} \end{bmatrix}\iota(\lambda_k)
    \iff (\forall g, h \in G) ~ x_{kg,kh}=x_{g,h}.
\]
Thus $x = \begin{bmatrix} x_{g,h} \end{bmatrix} \in (\A \rtimes_{\alpha,r} G)' \cap (\A \otimes B(\ell^2(G)))$ if and only if
\[
    x = \sum_g\sum_h (x_{h,hg} \otimes E_{h,hg}) = \sum_g\sum_h (x_{e,g} \otimes E_{h,hg})
    = \sum_g (x_{e,g} \otimes \rho_g),
\]
where $\alpha_g(a)x_{e,g}=x_{e,g}a$ for all $a \in \A$.
\end{proof}

Another advantage of the assumption $|G| < \infty$ is that we can compute relevant injective envelopes exactly. 

\begin{lemma} \label{inj_env_fin_grp_act}
Let $\alpha:G \curvearrowright \A$ be an action of a finite group on a unital $C^*$-algebra. Then $I(\A \rtimes_{\alpha,r} G) = I(\A) \rtimes_{I(\alpha),r} G$ and $I(\A^G)=I(\A)^G$.
\end{lemma}

\begin{proof}
As noted in Section \ref{inj},
\[
    \A \rtimes_{\alpha,r} G \subseteq I(\A) \rtimes_{I(\alpha),r} G \subseteq I(\A \rtimes_{\alpha,r} G).
\]
But $I(\A) \rtimes_{I(\alpha),r} G$ is injective, since it is the image of $I(\A) \otimes B(\ell^2(G))$ under a conditional expectation (Lemma \ref{fin_ind_CE}). Thus $I(\A) \rtimes_{I(\alpha),r} G = I(\A \rtimes_{\alpha,r} G)$. Now let $p=\frac{1}{|G|}\sum_g \lambda_g \in \Proj(\A \rtimes_{\alpha,r} G)$. Then by Lemma \ref{corner_iso} and \cite[Proposition 6.3]{Hamana1982I}
\begin{eqnarray*}
    I(\A^G)
    &\cong& I(p(\A \rtimes_{\alpha,r} G)p)\\
    &\cong& pI(\A \rtimes_{\alpha,r} G)p\\
    &\cong& p(I(\A) \rtimes_{I(\alpha),r} G)p\\
    &\cong& I(\A)^G.
\end{eqnarray*}
\end{proof}

Using the previous lemma, we can relate pseudo-expectations for $\A^G \subseteq \A$ and $\A \rtimes_{\alpha,r} G \subseteq \A \otimes B(\ell^2(G))$ with conditional expectations for $I(\A)^G \subseteq I(\A)$ and $I(\A) \rtimes_{I(\alpha),r} G \subseteq I(\A) \otimes B(\ell^2(G))$, respectively.

\begin{lemma} \label{PsExp_extends_to_CE}
Let $\alpha:G \curvearrowright \A$ be an action of a finite group on a unital $C^*$-algebra. Then every pseudo-expectation for $\A \rtimes_{\alpha,r} G \subseteq \A \otimes B(\ell^2(G))$ extends to a conditional expectation for $I(\A) \rtimes_{I(\alpha),r} G \subseteq I(\A) \otimes B(\ell^2(G))$. Likewise every pseudo-expectation for $\A^G \subseteq \A$ extends to a conditional expectation for $I(\A)^G \subseteq I(\A)$.  
\end{lemma}

\begin{proof}
By Lemma \ref{inj_env_fin_grp_act}, we may assume that $I(\A^G)=I(\A)^G$. Let $\Psi:\A \to I(\A)^G$ be a ucp map such that $\Psi|_{\A^G}=\id$. By injectivity, there exists $\overline{\Psi}:I(\A) \to I(\A)^G$ such that $\overline{\Psi}|_{\A}=\Psi$. Since $\overline{\Psi}|_{\A^G}=\Psi|_{\A^G}=\id$, the rigidity of the injective envelope allows us to conclude that $\overline{\Psi}|_{I(\A)^G}=\id$, which says $\overline{\Psi}$ is a conditional expectation. The proof for $\A \rtimes_{\alpha,r} G \subseteq \A \otimes B(\ell^2(G))$ is similar. 
\end{proof}

\subsubsection{The case of $\A$ simple}

We first consider $\A^G \subseteq \A$ when $\A$ is simple. If the action $\alpha:G \curvearrowright \A$ is (pointwise) outer, then there exists a unique pseudo-expectation. While this might seem natural, we remind the reader that having a unique pseudo-expectation is a ``largeness'' condition for the subalgebra, whereas $\A^G$ becomes ``smaller'' as the action becomes less trivial.

\begin{theorem} \label{simple_fixed}
Let $\alpha:G \curvearrowright \A$ be a (pointwise) outer action of a finite group on a unital simple $C^*$-algebra $\A$. Then $\A \rtimes_{\alpha,r} G \subseteq \A \otimes B(\ell^2(G))$ has a unique pseudo-expectation, which implies $\A^G \subseteq \A$ also has a unique pseudo-expectation.
\end{theorem}

\begin{proof}
To show that $\A \rtimes_{\alpha,r} G \subseteq \A \otimes B(\ell^2(G))$ has a unique pseudo-expectation, it suffices to show that $I(\A) \rtimes_{I(\alpha),r} G \subseteq I(\A) \otimes B(\ell^2(G))$ has a unique conditional expectation (Lemma \ref{PsExp_extends_to_CE}). By Lemma \ref{unique_CE_crossed}, this is equivalent to showing that
\[
    (I(\A) \rtimes_{I(\alpha),r} G)' \cap (I(\A) \otimes B(\ell^2(G))) = Z(I(\A) \rtimes_{I(\alpha),r} G).
\]
Now since $\alpha:G \curvearrowright \A$ is outer and $\A$ is simple, $\alpha$ is actually properly outer, which implies $I(\alpha):G \curvearrowright I(\A)$ is properly outer, therefore freely acting (see \cite[Remark 4.3 and Section 2.3]{Zarikian2019a}). Also $I(\A)$ is simple \cite[Proposition 4.15]{Hamana1979}. Thus by Lemma \ref{rel_comm_cross_prod},
\[
    (I(\A) \rtimes_{I(\alpha),r} G)' \cap (I(\A) \otimes B(\ell^2(G)))
    = Z(I(\A)) \otimes I = \bbC(1 \otimes I),
\]
as required. Now since $\A \rtimes_{\alpha,r} G$ is simple \cite[Theorem 3.1]{Kishimoto1981}, $p=\frac{1}{|G|}\sum_g \lambda_g \in \Proj(\A \rtimes_{\alpha,r} G)$ is full. Therefore $\A^G \subseteq \A$ has a unique pseudo-expectation, by Lemma \ref{corner_iso} and Theorem \ref{PZ}.
\end{proof}

Theorem \ref{simple_fixed} suggests the following general question, which lies just beyond the scope of our results.

\begin{question} \label{fixed_question}
If $G \curvearrowright \A$ is a pointwise-outer, strongly-continuous action of a compact group on a unital $C^*$-algebra, does $\A^G \subseteq \A$ have a unique pseudo-expectation?
\end{question}

In particular, we would very much like to know the answer to the following concrete version of Question \ref{fixed_question}.

\begin{question} \label{CAR}
If $\A = \bigotimes_{n=1}^\infty \bbM_2$ is the CAR algebra and $\alpha:\bbT \curvearrowright \A$ is the gauge action, does $\A^{\bbT} \subseteq \A$ have a unique pseudo-expectation? If so, does $\A^{\bbT} \subseteq \A$ have the AEP?
\end{question}

\begin{remark}
It is well-known that the CAR algebra is simple, and \cite[Theorem 5]{Araki1970} shows that gauge action is (pointwise) outer. Of course Theorem \ref{simple_fixed} doesn't apply, since $|\bbT|=\infty$. It is easy to see that $\A^{\bbT} \subseteq \A$ has a (faithful) unique conditional expectation, namely $E(a) = \int_{\bbT} \alpha_z(a)d\mu(z)$, where $\mu$ is the Haar measure. On the other hand, $\A^{\bbT} \subseteq \A$ fails the PEP, since \cite[Theorem 5.6]{BakerPowers1986} implies that every $\omega \in PS(\bbM_2)$ gives rise to $\rho = \otimes_{n=1}^\infty \omega \in PS(\A)$ with $\rho|_{\A^{\bbT}} \in PS(\A^{\bbT})$, but $\rho \circ \alpha_z \neq \rho$ if $\omega \circ \Ad\left(\begin{bmatrix} z & 0\\ 0 & \o{z} \end{bmatrix}\right) \neq \omega$. Since $\A$ is separable, if it turns out that $\A^{\bbT} \subseteq \A$ has a unique pseudo-expectation but fails the AEP, then one would have an example of a $C^*$-inclusion with a unique pseudo-expectation but which fails aperiodicity, answering the Question \ref{open} negatively.
\end{remark}

Sticking with simple $C^*$-algebras, we now consider (pointwise) inner actions $\alpha:G \curvearrowright \A$. We remind the reader that even though
\[
    (\forall g \in G)(\exists w_g \in U(\A)) ~ \alpha_g=\Ad(w_g),
\]
the map $w:G \to U(\A)$ need not be a representation, even if $G$ is abelian \cite[Example 5.6]{Phillips2009}. On the other hand, since $\A$ is simple, $w$ is a \emph{projective representation}, meaning that $w_gw_h=\sigma(g,h)w_{gh}$ for a 2-cocycle $\sigma:G \times G \to \bbT$. Although the following lemma is a standard result in group cohomology, we give an unsophisticated proof for the reader's convenience.

\begin{lemma} \label{cohomology}
Let $w:G \to U(\A)$ be a projective representation of a finite \ul{abelian} group. Then the following are equivalent:
\begin{enumerate}
\item[i.] $w_gw_h=w_hw_g$ for all $g, h \in G$;
\item[ii.] there exists a (genuine) representation $u:G \to U(\A)$ such that $u_g \in \bbC w_g$ for all $g \in G$.
\end{enumerate}
\end{lemma}

\begin{proof}
(i $\implies$ ii) If $G = \langle g \rangle \cong \bbZ_n$, then $w_g^n = \zeta1$ for some $\zeta \in \bbT$. Setting $u_{g^k}=(\zeta^{-1/n}w_g)^k$ for $k \in \bbZ$ does the job. In general, if $G=G_1G_2...G_n$ (internal direct product), where each $G_j$ is cyclic, then letting $u^{(j)}:G_j \to U(\A)$ be a (genuine) representation such that $u_{g_j}^{(j)} \in \bbC w_{g_j}$ for all $g_j \in G_j$, we have that $u:G \to U(\A)$ defined by
\[
    u_g = u_{g_1}^{(1)}u_{g_2}^{(2)}...u_{g_n}^{(n)}, ~ g = g_1g_2...g_n \in G_1G_2...G_n,
\]
does the job. Indeed, by assumption, $u_{g_j}^{(j)}u_{g_k}^{(k)}=u_{g_k}^{(k)}u_{g_j}^{(j)}$ for all $g_j \in G_j$ and $g_k \in G_k$. 

(ii $\implies$ i) Trivial.
\end{proof}

We will also need the following double relative commutant result, which is probably known.

\begin{lemma} \label{double_rel_commutant}
Let $\A$ be a unital simple $C^*$-algebra and $\U \subseteq \A$ be a unital finite-dimensional $C^*$-subalgebra. Then
\[
    (\U' \cap \A)' \cap \A = \U.
\]
\end{lemma}

\begin{proof}
First assume that $\U \cong \bbM_d$. Then the $C^*$-inclusion $\U \subseteq \A$ may be identified with the $C^*$-inclusion $\bbM_d \otimes 1 \subseteq \bbM_d \otimes (\U' \cap \A)$ \cite[Lemma 6.6.3]{KR2}. Now by \cite[Theorem 1]{HaydonWassermann1973},
\begin{eqnarray*}
    ((\bbM_d \otimes 1)' \cap (\bbM_d \otimes (\U' \cap \A)))' \cap (\bbM_d \otimes (\U' \cap \A))
    &=& (Z(\bbM_d) \otimes (\U' \cap \A))' \cap (\bbM_d \otimes (\U' \cap \A))\\
    &=& (I \otimes (\U' \cap \A))' \cap (\bbM_d \otimes (\U' \cap \A))\\
    &=& \bbM_d \otimes Z(\U' \cap \A).
\end{eqnarray*}
On the other hand, by \cite[Corollary 1]{HaydonWassermann1973},
\[
    \bbC(I \otimes 1) = Z(\bbM_d \otimes (\U' \cap \A)) = Z(\bbM_d) \otimes Z(\U' \cap \A) = 1 \otimes Z(\U' \cap \A),
\]
which implies $Z(\U' \cap \A) = \bbC1$. It follows that
\[
    ((\bbM_d \otimes 1)' \cap (\bbM_d \otimes (\U' \cap \A)))' \cap (\bbM_d \otimes (\U' \cap \A))
    = \bbM_d \otimes 1,
\]
which shows that $(\U' \cap \A)' \cap \A = \U$.\\

In general, there exist $z_1, z_2, ..., z_n \in \Proj(Z(\U))$ such that $\sum_{j=1}^n z_j=1$ and $\U z_j \cong \bbM_{d_j}$, $1 \leq j \leq n$. Then since $z_j\A z_j$ is simple \cite[Theorem 3.2.8]{MurphyBook}, the previous paragraph applies to show that
\[
    (\U' \cap \A)' \cap \A = \sum_{j=1}^n (((\U z_j)' \cap (z_j\A z_j))' \cap (z_j\A z_j)) = \sum_{j=1}^n \U z_j = \U.
\]
\end{proof}

With the help of Lemmas \ref{cohomology} and \ref{double_rel_commutant} above, we establish the following relative commutant result.

\begin{lemma} \label{rel_comm_inner}
Let $\alpha:G \curvearrowright \A$ be a (pointwise) inner action of a finite group on a unital simple $C^*$-algebra $\A$. Then
\[
    (\A \rtimes_{\alpha,r} G)' \cap (\A \otimes B(\ell^2(G)) = Z(\A \rtimes_{\alpha,r} G)
\]
if and only if $G$ is abelian and there exists a (genuine) representation $u:G \to U(\A)$ such that $\alpha_g = \Ad(u_g)$ for all $g \in G$.
\end{lemma}

\begin{proof}
Recall that
\[
    (\A \rtimes_\alpha G)' \cap (\A \otimes B(\ell^2(G))
    = \left\{\sum_g (d_g \otimes \rho_g): \begin{tabular}{c} $(\forall g \in G$) $d_g \in \A$ is a dependent\\ element for $\alpha_g \in \Aut(\A)$ \end{tabular}\right\}
\]
(Lemma \ref{rel_comm_cross_prod}). ($\Longrightarrow$) Fix $g \in G$ and let $w_g \in U(\A)$ be such that $\alpha_g = \Ad(w_g)$. Then $w_g$ is a dependent element for $\alpha_g$, and so $w_g \otimes \rho_g \in Z(\A \rtimes_{\alpha,r} G)$. Thus for all $g, h \in G$, we have that
\[
    w_gw_h \otimes \rho_{gh} = (w_g \otimes \rho_g)(w_h \otimes \rho_h) = (w_h \otimes \rho_h)(w_g \otimes \rho_g) = w_hw_g \otimes \rho_{hg}.
\]
It follows that $gh=hg$ and $w_gw_h=w_hw_g$ for all $g, h \in G$. Thus $G$ is abelian and (by Lemma \ref{cohomology}) there exists a representation $u:G \to U(\A)$ such that $\alpha_g=\Ad(u_g)$ for all $g \in G$.

($\Longleftarrow$) Since $G$ is abelian, $\rho_g = \lambda_g^*$. Now suppose $\alpha_g(a)d_g=d_ga$ for all $a \in \A$. Then $u_g^*d_g \in Z(\A) = \bbC1$, which implies $d_g \in \bbC u_g \subseteq \A^G$. Thus
\[
    \sum_g (d_g \otimes \rho_g) = \sum_g (d_g \otimes I)(1 \otimes \lambda_g^*) = \sum_g d_g\lambda_g^* \in \A \rtimes_{\alpha,r} G.
\]
\end{proof}

Finally we arrive at the main result.

\begin{theorem}
Let $\alpha:G \curvearrowright \A$ be a (pointwise) inner action of a finite group on a unital simple $C^*$-algebra $\A$. Fix any projective representation $w:G \to U(\A)$ such that $\alpha_g=\Ad(w_g)$ for all $g \in G$. Then the following statements hold:
\begin{enumerate}
\item[i.] $\A^G \subseteq \A$ has a unique pseudo-expectation $\iff$ $w_gw_h=w_hw_g$ for all $g, h \in G$.
\item[ii.] $\A \otimes_{\alpha,r} G \subseteq \A \otimes B(\ell^2(G))$ has a unique pseudo-expectation $\iff$ $G$ is abelian and $w_gw_h=w_hw_g$ for all $g, h \in G$ $\iff$ $G$ is abelian and there exists a (genuine) representation $u:G \to U(\A)$ such that $\alpha_g=\Ad(u_g)$ for all $g \in G$.
\end{enumerate}
In particular, (2) $\implies$ (1).
\end{theorem}

\begin{proof}
(i) By \cite[Theorem 3.2]{Izumi2002}, $\A^G \subseteq \A$ has finite Watatani index, since $\A$ is simple and $M_G(a^*a) \geq \frac{1}{|G|}a^*a$ for all $a \in \A$. Set $\W = \span\{w_g: g \in G\} \subseteq \A$. Since $w:G \to U(\A)$ is a projective representation, $\W \subseteq \A$ is a unital finite-dimensional $C^*$-subalgebra. By Lemma \ref{double_rel_commutant},
\[
    (\A^G)' \cap \A = (\W' \cap \A)' \cap \A = \W.
\]
By Theorem \ref{unique_CE_fin_ind}, $\A^G \subseteq \A$ has a unique conditional expectation if and only if $\W = Z(\A^G)$, if and only if $w_gw_h=w_hw_g$ for all $g, h \in G$. Since $I(\A)$ is simple and $I(\alpha)_g = \Ad(w_g)$ by uniqueness, the same argument shows that $I(\A)^G \subseteq I(\A)$ has a unique conditional expectation if and only if $w_gw_h=w_hw_g$ for all $g, h \in G$. Now if $\A^G \subseteq \A$ has a unique pseudo-expectation, then $\A^G \subseteq \A$ has a unique conditional expectation. And if $I(\A)^G \subseteq I(\A)$ has a unique conditional expectation, then $\A^G \subseteq \A$ has a unique pseudo-expectation, by Lemma \ref{PsExp_extends_to_CE}. Thus $\A^G \subseteq \A$ has a unique pseudo-expectation if and only if $w_gw_h=w_hw_g$ for all $g, h \in G$.\\

(ii) If $\A \rtimes_{\alpha,r} G \subseteq \A \otimes B(\ell^2(G))$ has a unique pseudo-expectation, then it has a unique conditional expectation, which implies $(\A \rtimes_{\alpha,r} G)' \cap (\A \otimes B(\ell^2(G))) = Z(\A \rtimes_{\alpha,r} G)$ (Lemma \ref{unique_CE_crossed}). By Lemma \ref{rel_comm_inner}, $G$ is abelian and there exists a representation $u:G \to U(\A)$ such that $\alpha_g=\Ad(u_g)$ for all $g \in G$. ($\Longleftarrow$) Conversely, suppose $G$ is abelian and there exists a representation $u:G \to U(\A)$ such that $\alpha_g=\Ad(u_g)$ for all $g \in G$. Then $I(\alpha)_g = \Ad(u_g)$ for all $g \in G$, by uniqueness. Since $I(\A)$ is simple, Lemma \ref{rel_comm_inner} implies $(I(\A) \rtimes_{I(\alpha),r} G)' \cap (I(\A) \otimes B(\ell^2(G))) = Z(I(\A) \rtimes_{I(\alpha),r} G)$. Thus $I(\A) \rtimes_{I(\alpha),r} G \subseteq I(\A) \otimes B(\ell^2(G))$ has a unique conditional expectation (Lemma \ref{unique_CE_crossed}), which implies $\A \rtimes_{\alpha,r} G \subseteq \A \otimes B(\ell^2(G))$ has a unique pseudo-expectation (Lemma \ref{PsExp_extends_to_CE}).
\end{proof}

\subsubsection{The case of $\A$ abelian}

Now we move to the opposite end of the spectrum and assume that $\A$ is abelian, i.e., $\A=C(X)$ for a compact Hausdorff space $X$. Then $\A^G \subseteq \A$ is isomorphic to $C(X/G) \hookrightarrow C(X): f \mapsto f \circ j$, where $j:X \twoheadrightarrow X/G: x \mapsto \orb_G(x)$ is the orbit map. Since $X/G$ is Hausdorff when $|G| < \infty$ \cite[Proposition 3.1]{tomDieckBook}, we can use Theorem \cite[Corollary 3.22]{PittsZarikian2015} to analyze this inclusion. We obtain the surprisingly negative result that $\A^G \subseteq \A$ has a unique pseudo-expectation if and only if $\alpha:G \curvearrowright \A$ is trivial. Recalling that $\A^G \subseteq \A$ is a corner of $\A \rtimes_{\alpha,r} G \subseteq \A \otimes B(\ell^2(G))$ (Lemma \ref{corner_iso}), we imagined that a negative answer to Question \ref{open_corner} was imminent. But it turns out that $\A \rtimes_{\alpha,r} G \subseteq \A \otimes B(\ell^2(G))$ also has multiple pseudo-expectations whenever $\alpha:G \curvearrowright \A$ is non-trivial.\\

Recall that a continuous surjection $j:Y \twoheadrightarrow X$ of compact Hausdorff spaces is \emph{irreducible} if $j(K) \subsetneq X$ whenever $K \subsetneq Y$ is closed. By \cite[Lemma 9A]{DaiXie2024}, this is equivalent to the condition that every nonempty open set $\emptyset \neq U \subseteq Y$ contains an entire fiber $j^{-1}(\{x_0\})$ for some $x_0 \in X$.

\begin{lemma} \label{irred_orb_map_general}
Let $\sigma:G \curvearrowright X$ be a continuous action of a finite group on a compact Hausdorff space. Then the orbit map $j:X \twoheadrightarrow X/G:x \mapsto \orb_G(x)$ is irreducible if and only if $\sigma$ is trivial.
\end{lemma}

\begin{proof}
First assume that $X$ is extremally disconnected. Suppose $\sigma_g:X \to X$ is non-trivial for some $e \neq g \in G$. By Frol\'{i}k's Theorem \cite[Theorem 3.1]{Frolik1971}, there exist disjoint clopen sets $F, C_1, C_2, C_3 \subseteq X$ such that $X=F \cup C_1 \cup C_2 \cup C_3$, $F=\fix(\sigma_g)$, and $\sigma_g(C_j) \cap C_j = \emptyset$ for all $j=1,2,3$. Since $\sigma_g$ is nontrivial, we may assume that $C_1 \neq \emptyset$. For any $x_0 \in C_1$, we have that $\sigma_g(x_0) \notin C_1$, which implies $\orb_G(x_0) \nsubseteq C_1$. Thus $j$ is not irreducible.\\

Now let $X$ be arbitrary. By Gleason's Theorem \cite[Theorem 3.2]{Gleason1958}, there exists an extremally-disconnected compact Hausdorff space $P$ and an irreducible continuous surjection $k:P \twoheadrightarrow X$. Since $C(P) \cong I(C(X))$, we have that $\sigma:G \curvearrowright X$ extends uniquely to a continuous action $\hat{\sigma}:G \curvearrowright P$ such that $k \circ \hat{\sigma}_g = \sigma_g \circ k$, $g \in G$. Now suppose that orbit map $j:X \twoheadrightarrow X/G$ is irreducible. A simple calculation shows that $j \circ k:P \twoheadrightarrow X/G$ is also irreducible. Now let $\emptyset \neq U \subseteq P$ be open. By irreducibility, there exists $x_0 \in X$ such that $k^{-1}(\orb_G(x_0)) \subseteq U$. But then $\orb_G(p_0) \subseteq U$ for any $p_0 \in k^{-1}(\{x_0\})$. In other words, the orbit map $P \twoheadrightarrow P/G: p \mapsto \orb_G(p)$ is irreducible. By the previous paragraph, $\hat{\sigma}$ is trivial. A fortiori, $\sigma$ is trivial.
\end{proof}

\begin{theorem} \label{abelian_fixed}
Let $\alpha:G \curvearrowright \A$ be a non-trivial action of a finite group on a unital abelian $C^*$-algebra. Then $\A^G \subseteq \A$ multiple pseudo-expectations. Likewise $\A \rtimes_{\alpha,r} G \subseteq \A \otimes B(\ell^2(G))$ has multiple pseudo-expectations.
\end{theorem}

\begin{proof}
Suppose $C(X/G) \subseteq C(X)$ has a unique pseudo-expectation. Then by \cite[Corollary 3.22]{PittsZarikian2015}, the orbit map $j:X \twoheadrightarrow X/G$ is irreducible. Thus by Lemma \ref{irred_orb_map_general}, $G \curvearrowright X$ is trivial.\\

Now suppose $\alpha:G \curvearrowright \A$ is non-trivial. Recall that
\[
    (\A \rtimes_\alpha G)' \cap (\A \otimes B(\ell^2(G))
    = \left\{\sum_g (d_g \otimes \rho_g): \begin{tabular}{c} $(\forall g \in G$) $d_g \in \A$ is a dependent\\ element for $\alpha_g \in \Aut(\A)$ \end{tabular}\right\}
\]
(Lemma \ref{rel_comm_cross_prod}). Since $\A$ is abelian, every $d \in \A$ is a dependent element for $\alpha_e=\id \in \Aut(\A)$. It follows that
\[
    \A \otimes I \subseteq (\A \rtimes_{\alpha,r} G)' \cap (\A \otimes B(\ell^2(G))).
\]
But
\[
    \A \otimes I \nsubseteq \A \rtimes_{\alpha,r} G,
\]
since $\alpha$ is non-trivial. Therefore
\[
    (\A \rtimes_{\alpha,r} G)' \cap (\A \otimes B(\ell^2(G))) \nsubseteq \A \rtimes_{\alpha,r} G,
\]
which implies $\A \rtimes_{\alpha,r} G \subseteq \A \otimes B(\ell^2(G))$ has multiple conditional expectations (Lemma \ref{unique_CE_crossed}).
\end{proof}

\subsubsection{The case of trivial actions}

For the sake of completeness, we record what happens when $\A$ is arbitrary but $\alpha:G \curvearrowright \A$ is trivial. In that case, $\A^G=\A$ (so $\A^G \subseteq \A$ has a unique pseudo-expectation) and $\A \rtimes_{\alpha,r} G = \A \otimes C_r^*(G)$.

\begin{theorem}
Let $\A$ be a unital $C^*$-algebra and $G$ be a finite group. Then the following are equivalent:
\begin{enumerate}
\item[i.] $\A \otimes C_r^*(G) \subseteq \A \otimes B(\ell^2(G))$ has a unique pseudo-expectation;
\item[ii.] $C_r^*(G) \subseteq B(\ell^2(G))$ has a unique pseudo-expectation;
\item[iii.] $G$ is abelian.
\end{enumerate}
\end{theorem}

\begin{proof}
Note that since $|G|<\infty$, $I(\A \otimes C_r^*(G)) = I(\A) \otimes C_r^*(G)$ by Lemma \ref{inj_env_fin_grp_act}. In particular, $I(C_r^*(G))=C_r^*(G)$. 

(i $\implies$ ii) Let $\theta:B(\ell^2(G)) \to C_r^*(G)$ be a conditional expectation. Then $\id \otimes \theta:\A \otimes B(\ell^2(G)) \to \A \otimes C_r^*(G)$ is a conditional expectation. Since $\id \otimes \theta$ is uniquely determined, so is $\theta$.

(ii $\implies$ iii) By Theorem \ref{finite_general}, $C_r^*(G)' \cap B(\ell^2(G)) = Z(C_r^*(G))$. Since $|G|<\infty$, $C_r^*(G)=L(G)$ and
\[
    C_r^*(G)' \cap B(\ell^2(G)) = L(G)' \cap B(\ell^2(G)) = R(G).
\]
But then $R(G)=Z(L(G))$, which implies $R(G)$ is abelian, which in turn implies $G$ is abelian.

(iii $\implies$ i) By Lemma \ref{PsExp_extends_to_CE}, it suffices to show that $I(\A) \otimes C_r^*(G) \subseteq I(\A) \otimes B(\ell^2(G))$ has a unique conditional expectation. But this follows from Lemma \ref{unique_CE_crossed}, since Lemma \ref{rel_comm_cross_prod} implies
\[
    (I(\A) \otimes C_r^*(G))' \cap (I(\A) \otimes B(\ell^2(G))) = Z(I(\A)) \otimes C_r^*(G).
\]
\end{proof}

\subsection{The Inclusion $Z(\A)^c \subseteq \A \rtimes_{\alpha,r} G$}

\phantom{}\\

In this final section, we consider the $C^*$-inclusion $Z(\A)^c \subseteq \A \rtimes_{\alpha,r} G$, where
\[
    Z(\A)^c = Z(\A)' \cap (\A \rtimes_{\alpha,r} G)
\]
is the relative commutant of the center of $\A$. Our real interest is in the $C^*$-inclusion $C^*(\A \cup \A^c) \subseteq \A \rtimes_{\alpha,r} G$, where $\A^c = \A' \cap (\A \rtimes_{\alpha,r} G)$ is the relative commutant of $\A$. But $C^*(\A \cup \A^c) \subseteq Z(\A)^c$, and we found the latter easier to work with. If $\A$ is abelian, then $Z(\A)^c = \A^c = C^*(\A \cup \A^c)$. And if $\A$ is simple, then $Z(\A)^c = \A \rtimes_{\alpha,r} G$. Given that $Z(\A)^c$ tends to be a ``large'' subalgebra of $\A \rtimes_{\alpha,r} G$, one would expect $Z(\A)^c \subseteq \A \rtimes_{\alpha,r} G$ to have a unique pseudo-expectation more often than not. Our results confirm this intuition.

\begin{lemma} \label{commutator}
Let $\A \subseteq \B$ be a $C^*$-inclusion with $\A$ \ul{abelian} and let $\Psi:\B \to I(\A^c)$ be a pseudo-expectation for $\A^c \subseteq \B$. Then $\Psi([b,a])=0$ for all $b \in \B$ and $a \in \A$.
\end{lemma}

\begin{proof}
Since $\A$ is abelian, $\A \subseteq \A^c$. Thus
\[
    \Psi([b,a]) = \Psi(ba-ab) = \Psi(b)a-a\Psi(b) = [\Psi(b),a].
\]
Now let $u \in U(\A)$. Define a ucp map $\iota:I(\A^c) \to I(\A^c): y \mapsto uyu^*$. Then $\iota(x)=x$ for all $x \in \A^c$, which implies $\iota(y)=y$ for all $y \in I(\A^c)$, by rigidity of the injective envelope. It follows that $uy=yu$ for all $y \in I(\A^c)$, which implies $ay=ya$ for all $a \in \A$ and $y \in I(\A^c)$. In particular, $[\Psi(b),a]=0$.
\end{proof}

For any $C^*$-algebra $\A$, the injective envelope $I(\A)$ is an $AW^*$-algebra. Its center $Z(I(\A))$ is an abelian $AW^*$-subalgebra, therefore injective \cite[Chapter 1, \S4 Proposition 8 and \S7 Theorem 1]{BerberianBook}. Also
\[
    Z(\A) = \A' \cap \A \subseteq \A' \cap I(\A) = Z(I(\A)),
\] 
by \cite[Corollary 4.3]{Hamana1979}. Thus $I(Z(\A)) \subseteq Z(I(\A))$. Unfortunately this inclusion need not be an equality. In fact, $I(Z(\A))=Z(I(\A))$ if and only if $Z(\A)$ \emph{detects regular ideals} in $\A$ (\cite[Corollary 1.6]{Hamana1982} and \cite[Theorems 6.3 and 6.6]{Hamana1981}). This means that $\J \cap Z(\A) \neq \{0\}$ whenever $\{0\} \neq \J \lhd \A$ and $\J^{\perp\perp}=\J$. Of course $I(Z(\A))=Z(I(\A))$ when $\A$ is abelian or simple. 

\begin{theorem} \label{center_commutant}
Let $(\A,G,\alpha)$ be a $C^*$-dynamical system. If $Z(\A)$ detects regular ideals in $\A$, then the $C^*$-inclusion $Z(\A)^c \subseteq \A \rtimes_{\alpha,r} G$ has the faithful unique pseudo-expectation property.
\end{theorem}

\begin{proof}
Recall from Section \ref{inj} that we have the following $C^*$-inclusions:
\[
    \begin{matrix}
        & & & & \\
        & & & & I_G(\A) \overline{\rtimes}_{I_G(\alpha)} G\\
        & & & & \cup\\
        \A \rtimes_{\alpha,r} G & \subset & I(\A) \rtimes_{I(\alpha),r} G & \subset & I_G(\A) \rtimes_{I_G(\alpha),r} G\\
        \cup & & \cup & & \cup\\
        \A & \subset & I(\A) & \subset & I_G(\A)\\
        \cup & & \cup & & \\
        Z(\A) & \subset & Z(I(\A)) & & \\
        & & & & \\
    \end{matrix}
\]
By \cite[Lemma 4.13]{Hamana1981} and \cite[Lemma 3.1]{Hamana1985} the inclusions
\[
    Z(I(\A)) \subset I(\A) \subset I_G(\A) \subset I_G(\A) \overline{\rtimes}_{I_G(\alpha)} G
\]
are \emph{normal} (preserve suprema of bounded increasing nets of self-adjoint elements). Since $Z(\A)^c \subseteq \A \rtimes_{\alpha,r} G \subseteq I_G(\A) \overline{\rtimes}_{I_G(\A)} G$ and $I_G(\A) \overline{\rtimes}_{I_G(\alpha)} G$ is injective, we may assume that $I(Z(\A)^c) \subseteq I_G(\A) \overline{\rtimes}_{I_G(\A)} G$ as an operator subsystem.\\

Now let $\Psi:\A \rtimes_{\alpha,r} G \to I(Z(\A)^c)$ be a pseudo-expectation for $Z(\A)^c \subseteq \A \rtimes_{\alpha,r} G$. Fix $g \in G$. Then $\alpha_g \in \Aut(\A)$ extends uniquely to $I(\alpha)_g \in \Aut(I(\A))$, which restricts to $\tilde{\alpha}_g \in \Aut(Z(I(\A)))$. By Frolik's Theorem \cite[Theorem 3.1]{Frolik1971}, there exist $p, q_1, q_2, q_3 \in \Proj(Z(I(\A)))$ such that $p+q_1+q_2+q_3=\one$, $\tilde{\alpha}_g(zp)=zp$ for all $z \in Z(I(\A))$, and $q_n\tilde{\alpha}_g(q_n)=0$ for $n=1,2,3$. Since $Z(\A)$ detects regular ideals in $\A$, we have that $Z(I(\A))=I(Z(\A))$, and by \cite[Theorem 6.6]{Hamana1981}, $I(Z(\A))=\overline{Z(\A)}$, where $\overline{Z(\A)}$ denotes the \emph{regular monotone completion}. Thus there exist increasing nets $\{e_i\}, \{r_j\}, \{s_k\}, \{t_\ell\} \subseteq Z(\A)_+$ such that 
\[
    p=\sup_{Z(I(\A))} e_i, ~ q_1=\sup_{Z(I(\A))} r_j, ~ q_2=\sup_{Z(I(\A))} s_k, ~ q_3=\sup_{Z(I(\A))} t_\ell.
\]
By normality,
\[
    p=\sup_{I_G(\A) \overline{\rtimes}_{I_G(\A)} G} e_i, ~ q_1=\sup_{I_G(\A) \overline{\rtimes}_{I_G(\A)} G} r_j, ~ q_2=\sup_{I_G(\A) \overline{\rtimes}_{I_G(\A)} G} s_k, ~ q_3=\sup_{I_G(\A) \overline{\rtimes}_{I_G(\A)} G} t_\ell.
\]
Now for all $z \in Z(\A)$,
\[
    \alpha_g(ze_i) = \tilde{\alpha}_g(ze_ip) = ze_ip = ze_i,
\]
which implies
\[
    (e_i\lambda_g)z = e_i\alpha_g(z)\lambda_g = \alpha_g(e_i)\alpha_g(z)\lambda_g = \alpha_g(e_iz)\lambda_g = \alpha_g(ze_i)\lambda_g = (ze_i)\lambda_g = z(e_i\lambda_g),
\]
which in turn implies $e_i\lambda_g \in Z(\A)^c$. Similarly,
\[
    r_j\alpha_g(r_j) = r_jq_1\tilde{\alpha}_g(r_jq_1) = r_jq_1\alpha_g(r_j)\tilde{\alpha}_g(q_1) = r_j\alpha_g(r_j)q_1\tilde{\alpha}_g(q_1) = 0,
\]
which implies
\[
    (r_j\lambda_g)^2 = r_j\lambda_gr_j\lambda_g = r_j\alpha_g(r_j)\lambda_g^2=0,
\]
which in turn implies $r_j\lambda_g \in \A \rtimes_{\alpha,r} G$ is a \emph{free normalizer} of $Z(\A)$ \cite[Definition 1]{Kumjian1986}. By \cite[Proof of Proposition 4]{Kumjian1986}, $r_j\lambda_g \in \overline{[\A \rtimes_{\alpha,r} G,Z(\A)]}$. Likewise $s_k\lambda_g, t_\ell\lambda_g \in \overline{[\A \rtimes_{\alpha,r} G,Z(\A)]}$. By Lemma \ref{commutator},
\[
    (e_i+r_j+s_k+t_\ell)\Psi(\lambda_g) = \Psi((e_i+r_j+s_k+t_\ell)\lambda_g) = e_i\lambda_g.
\]
Taking order limits in $I_G(\A) \overline{\rtimes}_{I_G(\A)} G$, we conclude that $\Psi(\lambda_g)=p\lambda_g$. Since $\Psi$ is an $\A$-bimodule map, it follows that $\Psi$ is uniquely determined. Moreover, $\ran(\Psi) \subseteq I(\A) \rtimes_{I(\alpha),r} G$.\\

Turning to the faithfulness of $\Psi$, we claim that $\bbE_{I(\A)} \circ \Psi = \bbE_{\A}$, where $\bbE_{\A}:\A \rtimes_{\alpha,r} G \to \A$ and $\bbE_{I(\A)}:I(\A) \rtimes_{I(\alpha),r} G \to I(\A)$ are the canonical faithful conditional expectations. Since both maps are $\A$-bimodular, it suffices to show that $\bbE_{I(\A)}(\Psi(\lambda_g))=\bbE_{\A}(\lambda_g)$ for all $g \in G$. Fix $g \in G$. By the previous discussion, $\Psi(\lambda_g)=p\lambda_g$, where $p \in \Proj(Z(I(\A)))$. If $g \neq e$, then
\[
    \bbE_{I(\A)}(\Psi(\lambda_g)) = \bbE_{I(\A)}(p\lambda_g) = p\bbE_{I(\A)}(\lambda_g) = 0 = \bbE_{\A}(\lambda_g).
\]
And if $g=e$, then $\lambda_g=\one$, which implies
\[
    \bbE_{I(\A)}(\Psi(\lambda_g)) = \bbE_{I(\A)}(\one) = \one = \bbE_{\A}(\lambda_g). 
\]
Using the claim,
\[
    \Psi(x^*x)=0 \implies \bbE_{I(\A)}(\Psi(x^*x))=0 \implies \bbE_{\A}(x^*x)=0 \implies x=0.
\]
\end{proof}

The following immediate consequence of Theorem \ref{center_commutant} answers \cite[Question 7]{PittsZarikian2015}.

\begin{corollary} \label{Q7 ans}
Let $(\A,G,\alpha)$ be a $C^*$-dynamical system, with $\A$ \ul{abelian}. Then $\A^c \subseteq \A \rtimes_{\alpha,r} G$ has a faithful unique pseudo-expectation.
\end{corollary}

We end with two questions.

\begin{question}
Can the hypothesis that $Z(\A)$ detects regular ideals in $\A$ be dropped from Theorem \ref{center_commutant}?
\end{question}

\begin{question}
Can $Z(\A)^c$ in Theorem \ref{center_commutant} be replaced by $C^*(\A \cup \A^c)$?
\end{question}

\subsection*{Acknowledgements}

The author was partially supported by a AMS-Simons Research Enhancement Grant for PUI (Primarily Undergraduate Institution) Faculty.

\subsection*{Disclaimer}

The views expressed in this article are those of the author and do not reflect the official policy or position of the U. S. Naval Academy, the Department of the Navy, the Department of War, or the U. S. Government.

\end{document}